\documentclass[letterpaper,12pt,reqno]{amsart}
\usepackage{bbm}
\usepackage[applemac]{inputenc}
\usepackage{amssymb}
\usepackage{amsfonts}
\usepackage{amsmath}
\usepackage[usenames]{color}
\usepackage{mathrsfs}
\usepackage{amsfonts}
\usepackage{amssymb,amsmath}
\usepackage{cite}
\usepackage{amsthm}
\usepackage{graphicx}
\usepackage{eucal}
\usepackage{tikz}

\def\la{\lambda}
\def\al{\alpha}

\def\b{\beta}

\def\N{\mathbb{N}}

\def\Q{\mathbb{Q}}

\numberwithin{equation}{section}
\newtheorem{theo}{Theorem}[section]
\newtheorem{coro}[theo]{Corollary}
\newtheorem{lemm}[theo]{Lemma}
\newtheorem{prop}[theo]{Proposition}

\theoremstyle{definition}
\newtheorem{defi}[theo]{Definition}

\newtheorem{remark}[theo]{Remark}

\numberwithin{equation}{theo}

\def\al{\alpha}

\def\b{\beta}
\def\l{\lambda}

\def\Z{\mathbb{Z}}

\def\j{{\jmath}}

\def\ZZ{\mathcal{Z}}
\def\bfU{{\mathbf U}}
\def\up{{\upsilon}}
\def\wK{{\widetilde K}}
\def\fS{{\mathfrak S}}

\def\GL{{\text{\rm GL}}}
\def\Sp{{\text{\rm Sp}}}

\def\O{{\text{\rm O}}}

\def\End{{\text{\rm End}}}
\def\sH{{\mathcal H}}
\def\bsH{{\boldsymbol{\mathcal H}}}
\def\sZ{{\mathcal Z}}
\def\sA{{\mathcal A}}
\def\sD{{\mathcal D}}

\def\sS{{\mathcal S}}

\def\sT{{\mathcal T}}
\def\fkm{{\mathfrak m}}
\def\ro{{\text{\rm ro}}}
\def\row{{\text{\rm row}}}
\def\co{{\text{\rm co}}}
\def\diag{{\text{\rm diag}}}

\def\wla{{\widetilde\la}}
\def\wmu{{\widetilde \mu}}

\def\La{{\Lambda}}
\def\bfj{{\mathbf j}}
\def\bsS{{\boldsymbol{\sS}}}

\def\SnrZ{{\sS_\sZ^\jmath(n,r)}}
\def\SnrA{{\sS_\sA^\jmath(n,r)}}

\def\SerrZ{{\sS^\imath(r,r)}}
\def\Xinr{{\Xi_{2n+1,2r+1}}}

\def\Xin{{\Xi_{2n+1}}}

\def\Xinl{{\Xi_{2n+1}^{0\diag}}}

\def\Xinri{{\Xi_{2n,2r}^\dagger}}

\def\Xini{{\Xi_{2n}^\dagger}}

\def\Xinli{{\Xi_{2n}^{\dagger\,0\diag}}}
\def\fkm{{\mathfrak m}}
\def\fkmAl{{\mathfrak m^{A,\mathbf 0}}}
\def\scmAl{{\textsc{m}^{A,\mathbf 0}}}
\def\bfl{{\mathbf 0}}

\def\bfe{{\mathbf e}}
\def\bft{{\mathbf t}}
\def\bfUjn{{\mathbf U^\jmath(n)}}
\def\bfUin{{\mathbf U^\imath(n)}}

\def\UinZ{{U_\sZ^\imath(n)}}
\def\UinR{{U_{\!R}^\imath(n)}}
\def\Om{{\Omega}}
\def\om{{\omega}}
\def\bsi{{\boldsymbol i}}

\def\bsq{{\boldsymbol q}}
\def\im{{\text{\rm im}}}
\def\lra{{\,\longrightarrow\,}}
\def\lmt{{\,\longmapsto\,}}

\def\SnriZ{{\sS^{\jmath\imath}(n,r)}}

\def\Xienr{{\Xi_{2n,2r}}}

\def\Xien{{\Xi_{2n}}}
\def\Xienrl{{\Xi_{2n,2r}^{0\diag}}}
\def\Xienl{{\Xi_{2n}^{0\diag}}}
\def\SenrQ{{\bsS^\imath(n,r)}}
\def\SenQ{{\bsS^\imath(n)}}
\def\bsSen{{\boldsymbol{\sS}^\imath(n)}}
\def\SenrZ{{\sS_\sZ^\imath(n,r)}}
\def\SenrA{{\sS_{\!\sA}^\imath(n,r)}}
\def\SenrR{{\sS_{\!R}^\imath(n,r)}}
\def\SnrA{{\sS_\sA^\jmath(n,r)}}

\def\SerrZ{{\sS_\sZ^\imath(r,r)}}

\def\hla{{\widehat\la}}
\def\hLa{{\widehat\La}}
\def\hmu{{\widehat \mu}}
\def\hnu{{\widehat \nu}}
\def\fkf{{\mathfrak f}}
\def\Ain{{\boldsymbol{\mathfrak{A}}^\imath(n)}}

\def\wt{\widetilde}
\def\wtt{{\text{\rm wt}}}
\def\im{{\text{\rm im}}}
\def\vep{{\varepsilon}}

\def\ev{{\text{\rm ev}}}
\def\sfm{{\mathsf m}}

\begin{document}

\title{The $i$-quantum group $\bfUin$}
\date{\today}

\author{Jie Du and Yadi Wu}

\address{J. D., School of Mathematics and Statistics,
University of New South Wales, Sydney NSW 2052, Australia}
\email{j.du@unsw.edu.au}
\address{Y. W. Institute of Mathematics, Academy of Mathematics and Systems Science, Chinese Academy of Sciences,  100190, Beijing, China}
\email{yadiwu@amss.ac.cn}


\date{\today}

\thanks{Part of the paper was written while the second author was visiting the University of New South Wales as a practicum student. She would like to thank UNSW for the hospitality and the China Scholarship Council for the financial support.}

\begin{abstract}This paper reveals some new structural property for the $i$-quantum group $\bfUin$ and constructs a certain hyperalgebra from the new structure which has connections to finite symplectic groups at the modular representation level. This work is built on certain finite dimensional $\mathbb Q(\up)$-algebras $\SenrQ$ whose integral form $\SenrZ$ is investigated as a convolution algebra arising from the geometry of type $C$  in \cite{BKLW}. We investigate $\SenrZ$ as an endomorphism algebra of a certain $q$-permutation module over the Hecke algebra of type $C_r$ and interpret the convolution product as a composition of module homomorphisms. We then prove that the action of $\bfUin$ on the $r$-fold tensor space of the natural representation of $\bfU(\frak{gl}_{2n})$ (via an embedding $\bfUin\hookrightarrow\bfU(\frak{gl}_{2n})$) coincides with an action given by multiplications in $\SenrQ$. In this way, we re-establish the surjective homomorphism from $\bfUin$ to $\SenrQ$ due to Bao--Wang \cite{BW}. We then embed $\bfUin$ into the direct product of $\SenrQ$ and completely determine its image. This gives a new realisation for $\bfUin$ and, as an application, the aforementioned hyperalgebra is an easy consequence of this new construction.


\end{abstract}

\maketitle

\tableofcontents

\section{Introduction}
Building on two fundamental results---some ``short'' multiplication formulas in \cite[Lem.~3.2]{BKLW} and a triangular relation between two bases in \cite[Thm.~3.10]{BKLW}---for the $q$-Schur algebras of type $B$, the authors successfully established certain ``long'' multiplication formulas which are used to develop a new presentation for the $i$-quantum group $\bfUjn$. This new realisation has an important application to a partial integral version of the Bao--Wang's Schur duality \cite[Thm.~6.27]{BW}. Thus,  this type of $q$-Schur algebras can now play a bridging role between the modular representation theory of the $i$-quantum groups $\bfUjn$ and that of finite orthogonal groups, a well-known connection in the type $A$ case established almost thirty years ago; see \cite{DJ89} and \cite{Du95}.

It is natural to expect that the $q$-Schur algebra $\SenrZ$ of type $C$ should play a similar bridging role between the $i$-quantum groups $\bfUin$ and finite symplectic groups. Since $\SenrZ$ is isomorphic to a centralizer subalgebra $e\SnrZ e$ for a certain idempotent $e$,
some of the structure of $\SenrZ$ 
such as the  aforementioned two fundamental results should be transferred from that of $e\SnrZ e$.
However, the ``twin products'' used in a triangular relation described in \cite[\S5.4]{BKLW} cannot be transferred. 
So, we must modify the triangular relation in $e\SnrZ e$ so that it can be transferred to $\SenrZ$. Also, in order to develop the link between representations of $\bfUin$, regarded as a coideal subalgebra of $\bfU(\mathfrak{gl}_{2n})$, and the finite symplectic groups, no transferring can apply; they all need to be developed independently.

This paper also follows an approach different from that used in \cite{DW}. After the definition of $\bfUin$, we identify $\bfUin$ as a coideal subalgebra of the quantum linear group $\bfU(\frak{gl}_{2n})$ and compute its actions on the tensor space $\Omega^{\otimes r}$, where $\Omega$ is the natural representation of $\bfU(\frak{gl}_{2n})$ (see Proposition \ref{U action}). Once we introduce the $q$-Schur algebra $\SenrZ$ of type $C$, we introduce an action on $\Omega^{\otimes r}$ by the  Hecke algebra $\bsH(r)$ via the vector space isomorphism from $\Omega^{\otimes r}$ to a $\SenrQ$-$\bsH(r)$-bimodule. We then identify the $\bfUin$-action on $\Omega^{\otimes r}$ with the $\SenrQ$-action (Theorem \ref{rH action}). In this way, we prove that the $\bfUin$-action commutes with the $\bsH(r)$-action and establish the partial Schur--Weyl duality in Corollary \ref{rho4}. This approach is different from the original proof given in \cite[Thm.~5.4]{BW} and is convenient to lift the partial duality to the integral level.

We organise the paper as follows. In \S2, we review the definition of $i$-quantum groups $\bfUin$ and their realisation as a coideal sublagebra of the quantum linear group $\bfU(\mathfrak{gl}_{2n})$. We then compute the action of $\bfUin$ on the tensor product $\Om^{\otimes r}$ via the coideal subalgebra embedding into $\bfU(\mathfrak{gl}_{2n})$. 
 In \S3, we introduce the $q$-Schur algebra $\SenrZ$ of type $C$ through the (Iwahori--)Hecke algebras of type $C_r$ and identify  $\Om^{\otimes r}$ as an $\SenrQ$-$\bsH(r)$-bimodule in \eqref{eta}. This allows us to transfer the $\bsH(r)$-action on the bimodule to an action on $\Om^{\otimes r}$. In \S4, short multiplication formulas in $\SenrZ$ (Lemma \ref{keyMF}) are introduced, following  \cite[Thm~3.7, Lem. A.13]{BKLW}, and are used to prove that the $\bfUin$-action on $\Om^{\otimes r}$ given in Proposition \ref{U action} coincides with the action of certain elements in $\SenrQ$ via multiplication formulas. This results in a map $\rho^\imath_{r}$ sending $\bfUin$ into $\End_{\bsH(r)}(\Om^{\otimes r})$ (Theorem \ref{rH action}). 
 In \S5, we develop a triangular relation between two bases of $\SenrZ$. This improves a similar relation via twin products in \cite[\S5.4]{BKLW}. Long multiplication formulas in $\SenrQ$ are derived in \S6 (Theorem \ref{longMF}). We further use them together with the triangular relation to prove that the map $\rho^\imath_{r}$ is surjective. As an application, we immediately get in \S7 a monomorphism from $\bfUin$ to the direct product $\SenQ$ of $\SenrQ$. The triangular relation further allows us to determine the image of the embedding. This gives us a new realisation of $\bfUin$ (Theorem \ref{summary}). Finally, in the last section, we introduce a $\sZ$-subalgebra $\UinZ$ so that $\rho_r^\imath$ induces an epimorphism from $\UinZ$ onto $\SenrZ$ and, at the same time, use the interpretation of $\SenrZ$ as the endomorphism algebra of a permutation module over the finite symplectic group $G=\Sp_{2r}(q)$ to establish a link between representations of $\UinR$ and $RG$ via $\SenrR$ for any field $R$.

\smallskip
\noindent
{\bf Some notations.} For a positive integer $a$, let
$$[1,a]=\{1,2,\ldots,a\},\quad [1,a)=\{1,2,\ldots,a-1\}.$$
Let $\sZ=\mathbb Z[\up,\up^{-1}]$ be the integral Laurent polynomial ring and let $\sA:=\mathbb Z[\bsq]$ ($\bsq=\up^{2}$). For any integers $n,m$ with $m>0$, we set
$$\left[\!\!\left[ n\atop m\right]\!\!\right]=\frac{\prod_{i=0}^{m-1}(\bsq^{(n-i)}-1)}{\prod_{i=1}^{m}(\bsq^{i}-1)}\in\sA,\text{ where }n\geq0,$$
$$[n]=\frac{\up^n-\up^{-n}}{\up-\up^{-1}}\quad \mbox{and} \quad \left[ n\atop m\right]=\frac{[n][n-1]\ldots[n-m+1]}{[1][2]\cdots[m]}=\prod_{i=1}^m\frac{\up^{n-i+1}-\up^{-(n-i+1)}}{\up^i-\up^{-i}}.$$
Denote $\left[\!\!\left[ n\atop 1\right]\!\!\right]$ as $[\![n]\!]$ and set $\left[\!\!\left[ n\atop 0\right]\!\!\right]=1=\left[ n\atop 0\right]$. Note that
\begin{equation}\label{0.n-m}\left[ n\atop m\right]=\up^{m(m-n)}\left[\!\!\left[ n\atop m\right]\!\!\right].
\end{equation}

 We also define, for $s,t\in\Z$ with $t>0$ and an element $K$ in a $\Q(\up)$-algebra,
\begin{equation}\label{Kbinom}
\left[\begin{matrix}K;s\\t\end{matrix}\right]=\prod_{i=1}^t\frac{K\up^{s-i+1}-K^{-1}\up^{-(s-i+1)}}{\up^i-\up^{-i}}.
\end{equation}

\section{The $i$-quantum group $\bfUin$ and its associated tensor modules}
In Bao-Wang's study \cite{BW} of canonical bases for the quantum symmetric pairs introduced in \cite{Le99, Le03}, they investigated two classes of such quantum symmetric pairs whose associated coideal subalgebras or $i$-quantum groups\footnote{Roughly speaking, an $i$-quantum group is a quantum analogue of the universal enveloping algebras of a fixed-point Lie subalgebra $\mathfrak g^\theta$ of a semisimple Lie algebra $\mathfrak g$ with an involution $\theta$. Here $i$ stands for {\it involution}.} are denoted by $\bfU^{\jmath}(n)$ and $\bfU^\imath(n)$, where $n$ indicates the rank of the $i$-quantum group. We now follow the definition of $\bfU^\imath(n)$ given in \cite[\S A.4]{BKLW} .
\begin{defi}\label{Uimath}
The algebra $\bfU^{\imath}(n)$ is defined to be the associative algebra over $\Q(\up)$ with generators  $e_i$, $f_i$, $d_a$, $d_{a}^{-1}$, $t$, for $i=1,2,\dots,n-1$, $a=1,2,\dots,n$ and the following relations\footnote{The relation $d_atd_a^{-1}=t$ is missing in \cite[(A.11)]{BKLW}.}: for $i,j=1,2,\dots,n-1$, $a,b=1,2,\dots,n$:
\begin{itemize}
\item[($\imath$QG1)] $d_ad_a^{-1}=d_a^{-1}d_a=1, d_ad_b=d_bd_a$;
\item[($\imath$QG2)] $d_ae_jd_a^{-1}=\up^{\delta_{a,j}-\delta_{a,j+1}}e_j,$
                     $d_af_jd_a^{-1}=\up^{-\delta_{a,j}+\delta_{a,j+1}}f_j$, $d_atd_a^{-1}=t$; 
\item[($\imath$QG3)] $e_if_j-f_je_i=\delta_{i,j}\frac{d_id_{i+1}^{-1}-d_i^{-1}d_{i+1}}{\up-\up^{-1}}$;
\item[($\imath$QG4)] $e_ie_j=e_je_i, f_if_j=f_jf_i,$ if $|i-j|>1$;
\item[($\imath$QG5)] $ e_i^2e_j+e_je_i^2=[2]e_ie_je_i,$ 
 $ f_i^2f_j+f_jf_i^2=[2]f_if_jf_i,$, if $|i-j|=1$;
\item[($\imath$QG6)]  \begin{itemize}
\item[(a)] $e_i t=t e_i$ for $i\neq n-1$, 
\item[(b)]$t^2 e_{n-1}+e_{n-1}t^2=[2]t e_{n-1} t+e_{n-1}$, 
\item[(c)]$e^2_{n-1}t+t e^2_{n-1}=[2]e_{n-1}t e_{n-1}$;
\end{itemize}
\item[($\imath$QG7)]\begin{itemize}
\item[(a)] $f_j t=t f_j$ for $j\neq n-1$, 
\item[(b)]$t^2 f_{n-1}+f_{n-1}t^2=[2]t f_{n-1}t+f_{n-1}$, 
\item[(c)]$f^2_{n-1} t+t f^2_{n-1}=[2]f_{n-1}t f_{n-1}$.
\end{itemize}
\end{itemize}
\end{defi}
We use the following diagram to indicate the relations of the generators.
 \begin{center}
\begin{tikzpicture}[scale=1.5]
\fill (0,0) circle (1.5pt);
\fill (1,0) circle (1.5pt);
\fill (2,0) circle (1.5pt);
\fill (4,0) circle (1.5pt);
\fill (5,0) circle (1.5pt);
  \draw (0,0) node[below] {$_1$} --
        (1,0) node[below] {$_2$} -- (2,0)node[below] {$_3$}--(2.5,0);
\draw[style=dotted](2.5,0)--(3.5,0);
\draw (3.5,0)--(4,0) node[below] {$_{n-1}$};
\draw[dashed] (4,0) --
        (5,0) node[below]  {$_t$};
\end{tikzpicture}\\
Figure 1.
\end{center}
Here the dashed line represents some unusual relations between $t$ and $e_{n-1}$ (resp., $f_{n-1}$) in ($\imath$QG6) (resp., ($\imath$QG7)).


The algebra $\bfUin$ admits an involution (i.e., algebra automorphism of order 2)
\begin{equation}\label{omega}
\om:\bfUin\lra\bfUin,\;\;e_i\lmt f_i,\; f_i\lmt e_i,\;d_a\lmt d_a^{-1},\;t\lmt t,
\end{equation}
 and an anti-involution
\begin{equation}\label{tau}
\tau:\bfUin\lra\bfUin,\;\;e_i\lmt e_i,\; f_i\lmt f_i,\;d_a\lmt d_a^{-1},\;t\lmt t;
\end{equation}
see \cite[Lem.~2.1]{BW}.

Consider the quantum linear group $\bfU(\mathfrak{gl}_{2n})$, a (Hopf) $\Q(\up)$-algebra defined by generators 
$$E_h,F_h,K_j^\pm,\;h\in[1,2n), j\in[1,2n].$$
and relations similar to ($\imath$QG1)--($\imath$QG5) with $E_h,F_h,K_j^\pm$ replacing $e_h,f_h,d_j^\pm$, respectively. Its comultiplication is defined by
\begin{equation}\label{Delta}
\Delta:\bfU(\mathfrak{gl}_{N})\longrightarrow \bfU(\mathfrak{gl}_{N})\otimes \bfU(\mathfrak{gl}_{N}),\;\;
\aligned
E_h&\longmapsto 1\otimes E_h+E_h\otimes \wK_h,\\
F_h&\longmapsto F_h\otimes 1+\wK_h^{-1}\otimes F_h,\\
K_j&\longmapsto K_j\otimes K_j.
\endaligned
\end{equation}
(We omit the counit and antipode maps as they are not used.)

The algebra $\bfU(\mathfrak{gl}_{2n})$ admits involution $\wt\om$
 similar to \eqref{omega} with $t$ omitted:
$$\wt\om: \bfU(\mathfrak{gl}_{2n})\lra \bfU(\mathfrak{gl}_{2n}),\;\;E_h\lmt F_h,\;F_h\lmt E_h,\;K_j\lmt K_j^{-1},$$
and an anti-involution 
$$\wt\tau:\bfU(\mathfrak{gl}_{2n})\lra \bfU(\mathfrak{gl}_{2n}),\;\;E_i\lmt E_i,\; F_i\lmt F_i,\;K_j\lmt K_j^{-1}$$ 
similar to \eqref{tau} with $t$ omitted.
(See, e.g., \cite[Lem.~6.5]{DDPW08}.)
We also need the ``graph automorphism'':
$$\wt\gamma: \bfU(\mathfrak{gl}_{2n})\lra \bfU(\mathfrak{gl}_{2n}),\;\;E_h\lmt E_{2n-h},\;F_h\lmt F_{2n-h},\;K_j\lmt K_{2n+1-j}^{-1}.$$

Note that
    $$\wt\gamma(\wK_i)=\wK_{2n-i},\;\text{ and }\; \wt\tau(\wK_i)=\wK_i^{-1}=\wt\om(\wK_i),$$
where $\wK_i=K_iK_{i+1}^{-1}$. 

The following realisation of $\bfUin$ is modified from \cite[Prop.~2.2]{BW}. We intentionally make the embedding to agree with the one given in \cite[Prop.~4.5]{BKLW}, thus, the embedding was chosen as below. 
\begin{lemm}\label{iota}
There is an injective $\Q(\up)$-algebra homomorphism $\iota:\bfUin\to\bfU(\mathfrak{gl}_{2n})$ defined, for $i\in[1,n),j\in[1,n]$, by 
$$\aligned
d_j\longmapsto K_j^{-1}K_{2n+1-j}^{-1},\;
e_i &\longmapsto F_i+\wK_i^{-1}E_{2n-i}, \; f_i\longmapsto E_i\wK_{2n-i}^{-1}+F_{2n-i},\\
t\,&\lmt F_n+\up^{-1}E_n\wK_n^{-1}+\wK_n^{-1}.\endaligned$$
Moreover, relative to the coalgebra structure, $\iota(\bfUin)$ is a coideal of $\bfU(\mathfrak{gl}_{2n})$ so that the comultiplication $\Delta$ in \eqref{Delta} restricts to an algebra homomorphism
$$\Delta:\iota(\bfUin)\lra\iota(\bfUin)\otimes \bfU(\mathfrak{gl}_{2n}).$$
\end{lemm}
\begin{proof}Let $\bfUin'$ be the $i$-quantum group generated by $t, e_{n+i}, f_{n+i}, d_{n+j}^{\pm1}$ for all
$i\in[1,n),j\in[1,n]$ obtained by an index shift $[-n+1,n-1]\to[1,2n),i\mapsto n+i$ from the $\bfU^\imath$ in \cite[p.24]{BW} (which is extended similarly from $\bfU(\mathfrak{sl}_{2n})$ to $\bfU(\mathfrak{gl}_{2n})$ as in \cite{BKLW}). Note that the associated diagram of $\bfUin'$ has the form
 \begin{center}
\begin{tikzpicture}[scale=1.5]
\fill (0,0) circle (1.5pt);
\fill (1,0) circle (1.5pt);
\fill (2,0) circle (1.5pt);
\fill (4,0) circle (1.5pt);
\fill (-1,0) circle (1.5pt);
  \draw (0,0) node[below] {$_{n+1}$} --
        (1,0) node[below] {$_{n+2}$} -- (2,0)node[below] {$_{n+3}$}--(2.5,0);
\draw[style=dotted](2.5,0)--(3.5,0);
\draw (3.5,0)--(4,0) node[below] {$_{2n-1}$};
\draw[dashed] (0,0) --
        (-1,0) node[below]  {$_t$};
\end{tikzpicture}\\
Figure 2.
\end{center}
Thus, after index shifting,
 the injective $\Q(\up)$-algebra homomorphism $\imath:\bfU^\imath\to \bfU$ in
\cite[Prop. 2.2]{BW} takes the form
$$\aligned
\imath:\bfUin'\to \bfU(\mathfrak{gl}_{2n}),\quad e_{n+i}&\lmt E_{n+i}+\wK_{n+i}^{-1}F_{n-i},\;f_{n+i}\lmt F_{n+i}\wK_{n-i}^{-1}+E_{n-i},\\
d_{n+j}&\lmt K_{n+j}K_{n+1-j},\; t\,\lmt E_n+vF_n\wK_n^{-1}+\wK_n^{-1},
\endaligned$$
 On the other hand, relabelling gives an algebra isomorphism
$$\gamma:\bfUin\lra\bfUin',\; t\,\lmt t, e_{n-i}\lmt e_{n+i}, \;f_{n-i}\lmt f_{n+i}, d_{n-j}\lmt d_{n+1+j}^{-1}.$$
Now, one checks easily that $\iota=\wt\om\circ\wt\tau\circ\wt\gamma\circ\imath\circ\gamma\circ\tau$. For example, for $i\in[1,n)$,
$$\aligned
e_{n-i}&\overset\tau\lmt e_{n-i}\overset\gamma\lmt e_{n+i}\overset\imath\lmt E_{n+i}+\wK_{n+i}^{-1}F_{n-i}\\
&\overset{\wt\gamma}\lmt E_{n-i}+\wK_{n-i}^{-1}F_{n+i}\overset{\wt\tau}\lmt E_{n-i}+\wK_{n-i}F_{n+i}
\overset{\wt\om}\lmt F_{n-i}+\wK_{n-i}^{-1}E_{n+i},
\endaligned$$
and, for $j\in[1,n]$, $\iota(d_{n+1-j})=\wt\om\circ\wt\tau\circ\wt\gamma\circ\imath(d_{n+j})=\wt\om\circ\wt\tau\circ\wt\gamma(K_{n+j}K_{n+1-j})=K_{n+j}^{-1}K_{n+1-j}^{-1}$,
as desired.
\end{proof}

\begin{remark}Dropping $\wt\om$ or $\wt\om$ together with $\wt\gamma$ and $\gamma$ in $\iota$ results in other embeddings. However, as shown in \cite[Thm.~7.1]{DW}, the embedding $\iota:\bfUin\to\bfU(\mathfrak{gl}_{2n})$ with the resulting action on the tensor space $\Omega^{\otimes r}$ is compatible with the action induced from $q$-Schur algebra at level $r$; see Theorem \ref{rH action} below. 
\end{remark}

Let $\Om=\Om_{2n}$ be the natural representation of $\bfU(\mathfrak{gl}_{2n})$ with a $\Q(\up)$-basis $\{\om_1,\om_2,\dots,\om_{2n}\}$ via the following actions:
\begin{equation}\label{NR}
E_h. \om_{i}=\delta_{i,h+1}\om_h,\;F_{h}.\om_{i}=\delta_{i,h}\om_{h+1},\; K_j.\om_i=\up^{\delta{i,j}}\om_i.
\end{equation}
Let 
\begin{equation}\label{I2nr}
I(2n,r):=\{\bsi=(i_1,\ldots,i_r)\mid i_j\in[1,2n]\}.
\end{equation}
For $\bsi=(i_1,\ldots,i_r)\in I(2n,r)$, let
 \begin{equation}\label{hat i}
 \widehat\bsi=(i_1,\ldots,i_r,i_{r+1},\ldots,i_{2r})\in I(2n,2r)
 \end{equation}
  be defined by setting $i_{2r+1-j}=2n+1-i_j$ for all $j\in[1,r]$. Define
\begin{equation}\label{weight}
\aligned
\wtt(\widehat\bsi)&=(\la_1,\la_2,\ldots,\la_n\ldots, \la_{2n}),\text{ where }\la_j=\#\{a\in[1,2r]\mid i_a=j\},\\
\wtt(\bsi)&=(\la_1,\la_2,\ldots,\la_{n}).\endaligned
\end{equation}

The {\it tensor space} $\Om^{\otimes r}$ has a basis 
$$\{\om_\bsi:=\om_{i_1}\otimes\om_{i_2}\otimes\cdots\otimes\om_{i_r}\mid\bsi=(i_1,\ldots,i_r)\in I(2n,r)\}$$
and becomes a $\bfU(\mathfrak{gl}_{2n})$-module via the actions:
\begin{equation*}\label{NR1}
E_h .\om_\bsi=\Delta^{(r-1)}(E_h)\om_\bsi,\;F_{h}.\om_\bsi=\Delta^{(r-1)}(F_h)\om_\bsi,\; K_j.\om_\bsi=\Delta^{(r-1)}(K_j)\om_\bsi,
\end{equation*}
where 
$$\Delta^{(r-1)}=(\Delta\otimes \underbrace{1\otimes\cdots\otimes 1}_{r-2})\circ\cdots\circ(\Delta\otimes 1)\circ\Delta:\bfU(\mathfrak{gl}_{2n})\lra \bfU(\mathfrak{gl}_{2n})^{\otimes r}.$$
Thus, we obtain a $\mathbb Q(\up)$-algebra homomorphism 
\begin{equation}\label{rho}
\rho_{r}=\rho_{n,r}:\bfU(\mathfrak{gl}_{2n})\lra\End(\Om_{2n}^{\otimes r}).
\end{equation} 

For any $\bsi=(i_1,\ldots,i_r)\in I(2n,r)$, we use the abbreviation 
$\om_\bsi:=\om_{i_1}\om_{i_2}\cdots\om_{i_r}$ below for $\om_{\bsi}=\om_{i_1}\otimes\om_{i_2}\otimes\cdots\otimes\om_{i_r}$ and call $\wtt(\bsi)$ the {\it weight} of $\om_\bsi$.
\begin{prop}\label{U action}
The $\mathbb Q(\up)$-algebra homomorphism
\begin{equation}\label{rhoi}
\rho_{r}^\imath:=\rho_{r}\circ\iota:\bfUin\overset\iota\lra\bfU(\mathfrak{gl}_{2n})\overset{\rho_{r}}\lra\End(\Om_{2n}^{\otimes r}) 
\end{equation}
defines a $\bfUin$-module structure on $\Om_{2n}^{\otimes r}$ which is given by the following action formulas: for all $\bsi=(i_1,\ldots,i_r)\in I(2n,r)$,
\begin{itemize}
\item[(1)] $\rho_{r}^\imath(d_j).\om_\bsi= \up^{-\delta(j,\bsi)}\om_\bsi,$ where 
$$\delta(j,\bsi)=|\{k\mid 1\leq k\leq r,i_k=j\}|+|\{k\mid 1\leq k\leq r,i_k=2n+1-j\}|.$$
\item[(2)] $\rho_{r}^\imath(e_h).\om_\bsi=\displaystyle\sum_{1\leq l\leq r\atop i_l=h} \up^{\vep_1(l)}\om_{i_1}\dots \om_{i_{l-1}} \om_{h+1} \om_{i_{l+1}}\dots  \om_{i_r}\\
\text{\qquad\qquad}+ \sum_{1\leq l\leq r\atop i_l=2n-h+1}\up^{\vep_1(l)+\vep_2(l)}\om_{i_{1}}\dots \om_{i_{l-1}}\cdot\om_{2n-h}\cdot \om_{i_{l+1}}\dots \om_{i_r}$,\\ where 
$$\aligned
\vep_1(l)&=-|\{k\mid 1\leq k<l,i_k=h\}|+|\{k\mid 1\leq k<l,i_k=h+1\}|\\
\vep_2(l)&=-|\{k\mid l< k\leq r,i_k=h\}|+|\{k\mid l<k\leq r,i_k=h+1\}|\\
      &\quad+\ |\{k\mid l< k\leq r,i_k=2n-h\}|-|\{k\mid l< k\leq r,i_k=2n-h+1\}|.\endaligned$$
\item[(3)] $\rho_{r}^\imath(f_h).\om_\bsi=\displaystyle\sum_{1\leq l\leq r\atop i_l=h+1}\up^{\vep'_1(l)+\vep'_2(l)}\om_{i_{1}}\dots \om_{i_{l-1}} \cdot\om_{h}\cdot\om_{i_{l+1}}\dots \om_{i_r},\\
\text{\qquad \qquad}+ \sum_{1\leq l\leq r\atop i_l=2n-h}\up^{\vep'_1(l)}\om_{i_1}\dots \om_{i_{l-1}} \om_{2n-h+1} \om_{i_{l+1}}\dots  \om_{i_r}$,\\ where 
$$\aligned
\vep'_1(l)&=-|\{k\mid 1\leq k< l,i_k=2n-h\}|+|\{k\mid 1\leq  k< l,i_k=2n-h+1\}|\\
      \vep'_2(l)&=|\{k\mid l<k\leq r,i_k=h\}|-|\{k\mid l<k\leq r,i_k=h+1\}|\\
      &\quad-\ |\{k\mid l< k\leq r,i_k=2n-h\}|+|\{k\mid l< k\leq r,i_k=2n-h+1\}|.\endaligned$$
\item[(4)] $\rho_{r}^\imath(t).\om_\bsi=\up^{\tau_0}\om_{i_1}\cdots \om_{i_{l-1}}\om_{i_l}\om_{i_{l+1}}\cdots \om_{i_r}+\displaystyle\sum_{1\leq l\leq r\atop i_l=n} \up^{\tau_1(l)}\om_{i_1}\cdots \om_{i_{l-1}}\om_{n+1}\om_{i_{l+1}}\cdots \om_{i_r}\\
\text{\qquad \qquad\;\;\;}
+\ \sum_{\substack{l=1\\i_l=n+1}}\up^{\tau_1(l)}\om_{i_1}\cdots \om_{i_{l-1}}\om_{n}\om_{i_{l+1}}\cdots \om_{i_r}$,\\
where $$\aligned
\tau_0&=|\{k\mid 1\leq k\leq r,i_k=n+1\}|-|\{k\mid 1\leq k\leq r,i_k=n\}|\\
\tau_1(l)&=|\{k\mid 1\leq k<l,i_k=n+1\}|-|\{k\mid 1\leq k< l,i_k=n\}|.\endaligned$$
\end{itemize}
\end{prop}
\begin{proof}We first compute  $\Delta^{(r-1)}(E_h)$, $\Delta^{(r-1)}(F_h)$, and $\Delta^{(r-1)}(K_j)$ as in \cite[(2.3.2)]{DW} and then apply the algebra homomorphism $\Delta^{(r-1)}$ to $\iota(d_j)=K_j^{-1}K_{2n+1-j}^{-1}$,
$\iota(e_i)= F_i+\wK_i^{-1}E_{2n-i}$, $\iota(f_i)= E_i\wK_{2n-i}^{-1}+F_{2n-i},$ and
$\iota(t)= F_n+\up^{-1}E_n\wK_n^{-1}+\wK_n^{-1}$. The remaining calculation via \eqref{NR} is straightforward. See the proof of \cite[Thm. 7.1]{DW} for action formulas (1)--(3), excluding the $h=n$ case there. For (4), we have \begin{eqnarray*}
\begin{aligned}
\rho_{r}^\imath(t).\om_{\bsi} &=\Delta^{(r-1)}(F_n+\up^{-1}E_n \wK^{-1}_{n}+\wK^{-1}_{n}).(\om_{i_1} \om_{i_2}\dots \om_{i_{r}})\\
&\ =\ \sum_{l=1}^{r} \wK_n^{-1}\om_{i_1}\cdots \wK_{n}^{-1}\om_{i_{l-1}} F_n\om_{i_{l}} \cdot\om_{i_{l+1}}\dots \om_{i_r} +\wK_{n}^{-1}\om_{i_1}\cdots \wK^{-1}_n \om_{i_r}\\
&\quad +\ \up^{-1}\sum_{l=1}^{r}\wK^{-1}_{n}\om_1\cdots \wK^{-1}_{n}\om_{i_{l-1}}.E_n\wK^{-1}_n \om_{i_{l}}. \om_{i_{l+1}}\cdots \om_{i_{r}}.
\end{aligned}
\end{eqnarray*} 
Now applying \eqref{NR} gives the desired formula.
\end{proof}
It is proved in \cite[Thm.~5.4]{BW} that, for the Hecke algebra $\bsH(r)$ of type $B_r$, there is a $\bfUin$-$\bsH(r)$-bimodule structure on $\Om_{2n}^{\otimes r}$ such that im$(\rho_{r}^\imath)=\End_{\bsH(r)}(\Om_{2n}^{\otimes r}) $. We will provide a different proof later in \S6 as a by-product of the multiplication formulas developed in \S4 (see Theorem \ref{rH action} and Corollary \ref{rho4}).

\section{The $q$-Schur algebra of type $C$}
The Weyl group of type $C_r$ is isomorphic to the Weyl group of type $B_r$ which is
 the Coxeter system $(W,S)$, where $S=\{s_1,\ldots,s_{r-1},s_r\}$ with the subgroup $W':=\langle s_1,\ldots,s_{r-1}\rangle\cong\fS_r$, the symmetric group on $r$ letters, and $s_{r-1}s_r$ has order 4.
In this case, $W$ is regarded as a fixed-point subgroup of $\fS_{2r}$ under the graph automorphism.
More precisely, there is a type $C$ embedding (cf. the type B embedding in \cite[\S3]{DW}): 
\begin{equation}\label{sigma}
\sigma:W\lra \fS_{2r},\;s_1\mapsto(1,2)(2r,2r-1),\ldots,s_{r-1}\mapsto(r-1,r)(r+2,r+1), s_r\mapsto(r,r+1).
\end{equation}
Then im$(\sigma)$ is the fixed-point subgroup $(\fS_{2r})^\theta$ of the involution $\theta$ on $\fS_{2r}$ sending $(i,j)$ to $(2r+1-i,2r+1-j)$. In other words, $W$ may be identified as the subgroup of $\fS_{2r}$ consisting of permutations
$$w=\begin{pmatrix}1&2&\cdots&r&r+1&\cdots&2r\\ i_1&i_2&\cdots&i_r&i_{r+1}&\cdots&i_{2r}\end{pmatrix}$$
satisfying $i_j+i_{2r+1-j}=2r+1$. 

If $W$ is regarded as a fixed-point subgroup of $\fS_{2r+1}$ (see Remark \ref{BvsC} below),  $W$ is called the Weyl group of type $B_r$. See  \cite{DW} for more details in this case and also \cite{LW} in general.
  
Let $\sH_\sA(r)=\sH_\sA(C_r)$ be the Hecke algebra over $\sA:=\mathbb Z[\bsq]$ ($\bsq=\up^{2}$) associated with$(W,S)$. 
 Then it is generated by $T_i=T_{s_i}$ for $1\leq i\leq r$ subject to the relations:
 $$\aligned
 &T_i^2=(\bsq-1)T_i+\bsq, \forall\ i; \;\; T_iT_j=T_jT_i, |i-j|\geq2,\\
 &T_jT_{j+1}T_j=T_{j+1}T_jT_{j+1}, 1\leq j<r-1;\\
 &T_{r-1}T_rT_{r-1}T_r=T_rT_{r-1}T_rT_{r-1}.
 \endaligned$$
 Let $\sZ=\Z[\up,\up^{-1}]$. We will mainly use the $\sZ$-algebra $\sH_\sZ(r)=\sH_\sA(r)\otimes \sZ$ in the sequel. Both $\sH_\sA(r)$ and $\sH_\sZ(r)$ have basis $\{T_w\}_{w\in W}$. The subalgebra generated by $T_1,\ldots,T_{r-1}$ is the Hecke algebra $\sH_\sA(\fS_r)$ over $\sA$ or $\sH_\sZ(\fS_r)$ over $\sZ$ associated with the symmetric group $\fS_r$.
 
 Let $\ell:W\to\mathbb N$ be the length function relative to $S$. Then, for $s\in S, w\in W$, we have
 $$T_sT_w=\begin{cases}T_{sw},&\text{ if }\ell(sw)=\ell(w)+1;\\
 (\bsq-1)T_w+\bsq T_{sw},&\text{ if }\ell(sw)=\ell(w)-1.\end{cases}$$

Let 
$$\Lambda(n,r)=\{\la=(\la_1,\la_2\ldots,\la_{n})\in \mathbb N^{n}\mid \la_1+\cdots+\la_{n}=r\}.$$ 
For $\la\in\Lambda(n,r)$, let 
$W_\la$ be the parabolic subgroup of $W$ generated by 
$$ S\backslash\{s_{\la_1+\cdots+\la_i}\mid i\in[1,n]\}.$$
Note that $W_\la$ is a subgroup of $W'$. Let
\begin{equation}\label{Tnr}
x_\la=\sum_{w\in W_\la}T_w,\quad \sT_\sZ(n,r)=\bigoplus_{\la\in\Lambda(n,r)}x_\la\sH_\sZ(r).
\end{equation}
 The endomorphism $\sZ$-algebra:
\begin{equation}\label{Snr}
\SenrZ=\End_{\sH_\sZ(r)}\big(\sT_\sZ(n,r)\big)
\end{equation}
is called the (generic) {\it $q$-Schur algebra of type $C$}. Note that $\SenrZ$ has also an $\sA$-form
$\SenrA$ defined by using $\sH_\sA(r)$ as above.

For $\la\in\La(n,r)$, let $\sD_\la$ be the set of shortest representatives of right cosets of $W_\la$ in $W$, and let $\sD_{\la\mu}=\sD_\la\cap\sD_\mu^{-1}$. Then $\sD_{\la\mu}$ is the set of shortest representatives of $W_\la$-$W_\mu$ double cosets.

Define
\begin{equation}\label{hla}
\widehat{\ }:\La(n,r)\lra\La(2n,2r),\;\;\la=(\la_1,\ldots,\la_n)\longmapsto\hla=(\la_1,\ldots,\la_n,\la_n,\ldots,\la_1).
\end{equation}
Note that, for any $\bsi\in I(2n,r)$ with $\widehat\bsi$ defined in \eqref{hat i}, $\wtt(\widehat\bsi)=\hla$ for some $\la\in\Lambda(n,r)$. In other words,
$\wtt(\widehat\bsi)=\widehat{\wtt(\bsi)}$.

For a positive integer $N$ and integer $r\geq0$, let 
\begin{equation}\label{Xi}
\aligned
\Xi_{N}&=\big\{A=(a_{i,j})\in \text{Mat}_N(\mathbb{N}) \mid a_{i,j}=a_{N+1-i,N+1-j}, \forall i,j\in [1,N]\},\\
\Xi_{N}^{0\diag}&=\big\{A-\diag(a_{1,1},a_{2,2},\ldots,a_{N,N})\mid A=(a_{i,j})\in\Xi_{N}\},\\
\Xi_{N,r}&=\{A=(a_{i,j})\in\Xi_{N}\mid |A|:=\sum_{i,j}a_{i,j}=r\},\;\text{ and }\;\Xi_{N,r}^{0\diag}=\Xi_{N}^{0\diag}\cap\Xi_{N,r}.
\endaligned
\end{equation}
Note that if we represent $w\in \fS_{2r}$ by a $2r\times 2r$ permutation matrix $P(w)=(p_{k,l})$, where $p_{k,l}=\delta_{k,i_l}$, then $w\in W=\fS_{2r}^\theta$ if and only if $P(w)\in\Xi_{2r,2r}$. 

For an $N\times N$ matrix $A=(a_{i,j})$, let
\begin{equation*}
\aligned
\ro(A)&:=\big(\sum_{j}a_{1,j},\,\sum_{j}a_{2,j},\dots,\sum_{j}a_{N,j}\big)\\
\co(A)&:=\big(\sum_{i}a_{i,1},\,\sum_{i}a_{i,2},\dots,\sum_{i}a_{i,N}\big).
\endaligned
\end{equation*}
Clearly, we have
$$\{\ro(A)\mid A\in\Xi_{N,2r}\}=\{\co(A)\mid A\in\Xi_{N,2r}\}=\widehat\La(n,r):=\{\hla\mid \la\in\La(n,r)\}.$$

\begin{lemm}\label{Snr basis}
(1) There is a bijection 
$$\fkm: \{(\la,d,\mu)\mid \la,\mu\in\Lambda(n,r), d\in\sD_{\la,\mu}\}\longrightarrow \Xi_{2n,2r}.$$

(2) The $\sZ$-algebra $\SenrZ$ is a free $\sZ$-module with basis 
$$\{\phi_A=\phi^d_{\la\mu}\mid \la,\mu\in\La(n,r),d\in\sD_{\la\mu}\},$$
where $\phi^d_{\la\mu}$ is defined by $\phi^d_{\la\mu}(x_\nu)=\delta_{\mu,\nu}T_{W_\la dW_\mu}$ and $A=\fkm(\la,d,\mu)$.
\end{lemm}
\begin{proof} 
For assertion (1), note that the matrix $A=\fkm(\la,d,\mu)$ is the matrix associated with the double coset $\fS_{\hla}\widehat d\fS_{\hmu}$ in $\fS_{2r}$, where $\sigma(W_\nu)=(\fS_{\widehat \nu})^\theta$ for $\nu=\lambda$ or $\mu$, and $\widehat d=\sigma(d)$. For more details, see, e.g., \cite[Lem.~2.2.1]{LL}.
 Assertion (2) follows from (1) and \cite[1.4]{Du94}.
\end{proof}

For any $A=\fkm(\la, d,\mu)$, let
$$[A]=\up^{-\ell(d^+)+\ell(w_{\mu,0})}\phi_{\la,\mu}^d,\;\;[\la]=\phi_{\la,\la}^1$$
where $d^+$ (resp. $w_{\mu,0}$) is the longest element in the double coset $W_\la dW_\mu$ (resp. $W_\mu$). Here we follow the definition given in \cite[1.4]{Du94} or \cite[(3.22)]{LW}\footnote{The Hecke algebra there is the Hecke algebra here with $\up$ replaced by $\up^{-1}$.} (cf. also \cite[(9.3.1)]{DDPW08}).  

Note that, for $A,B\in\Xienr$,
\begin{equation}\label{weight idemp}
[A][B]\neq0\implies \co(A)=\ro(B)\;\; \text{ and }\;\;[\ro(A)][A]=[A]=[A][\co(A)].
\end{equation}

If $n\geq r$, then the basis element $e_\emptyset:=[\diag(\emptyset)]\in\SenrZ$ is an idempotent, where
$$\emptyset=(\underbrace{1,\ldots,1}_r,\underbrace{0,\ldots,0}_{n-r},\underbrace{0\ldots,0}_{n-r},\underbrace{1,\ldots,1}_r)\in\Xi_{2n,2r},$$
and $e_\emptyset\SenrZ e_\emptyset\cong \sH_\sZ(r)$ via the evaluation map $\phi_{\emptyset\emptyset}^w\mapsto \phi_{\emptyset\emptyset}^w(1)=T_w$ for all $w\in W$. Via this isomorphism, $\SenrZ e_{\emptyset}$ becomes  an $\SenrZ$-$\sH_\sZ(r)$-bimodule. 

For  $\bsi=(i_1,i_2,\dots,i_r)\in I(2n,r)$ and $A_\bsi=(a_{k,l})\in\Xi_{2n,2r}$ defined by
\begin{equation}\label{Ai}
a_{k,l}=
\begin{cases}
\delta_{k,i_l}, &\text{ if }l\in[1,r];\\
0, &\text{ if }l\in[r+1,2n-r];\\
\delta_{2n+1-k,i_{2n+1-l}} &\text{ if } l\in[2n+1-r, 2n].\\
\end{cases}
\end{equation}
Note that $\co(A_\bsi)=(1^r,0^{n-r},0^{n-r},1^r)=\emptyset$ and $\ro(A_\bsi)=\wtt(\widehat{\bsi^*})$, where
$\bsi^*=(\bsi,0^{n-r})$ and $\widehat{\bsi^*}$ is similarly defined as in \eqref{hat i} with $i_{2n+1-j}=2n+1-i_j$. If we write $\widehat{\bsi^*}=(i_1,\ldots,i_r,i_{r+1},\ldots.i_{2n})$, then the entry $a_{k,l}$ of $A_\bsi$ has the form $a_{k,l}=\delta_{k,i_l}$ for all $k,l\in[1,2n]$.

By Lemma \ref{Snr basis}(1), there exist $\la=\la_\bsi\in\La(n,r)$, $d=d_\bsi\in\sD_\la$ such that $A_\bsi=\fkm(\la,d,\emptyset)$. Thus, under the assumption $n\geq r$, the evaluation map
\begin{equation}\label{ev}
\ev:\SenrZ e_\emptyset\lra\sT_\sZ(n,r),\;\;[A_\bsi]\lmt \up^{-\ell(d_\bsi^+)}x_{\la_\bsi}T_{d_\bsi}(=\up^{-\ell(d_\bsi^+)}\phi_{\la_\bsi,\emptyset}^{d_\bsi}(1))
\end{equation}
defines an  $\SenrZ$-$\sH_\sZ(r)$-bimodule isomorphism.


If $n<r$, then we may identify $\Xi_{2n,2r}$ as a subset of $\Xi_{2r,2r}$ via the following embedding:
\begin{equation}\label{Acirc}
\Xi_{2n,2r}\lra \Xi_{2r,2r},\;\;A\longmapsto A^\circ=\begin{pmatrix}0&0&0\\0& A&0\\
0&0&0\\\end{pmatrix},
\end{equation}
where each 0 at a corner position of $A^\circ$ is a square zero matrix of size $r-n$ and other zeros  represent zero matrices of appropriate sizes. 
Thus, if $n<r$, we may regard $\SenrZ$ as a centraliser subalgebra of $\SerrZ$ via the induced embedding
$[A]\mapsto[A^\circ]$.

Note also that the embedding $A\mapsto A^\circ$ induces an embedding 
\begin{equation}\label{la^o}
\La(n,r)\lra\La(r,r),\la\longmapsto \la^\circ:=(0^{r-n},\la).
\end{equation}
In $\SerrZ$, the idempotent $f=\sum_{\l\in\La(n,r)}[\hla^\circ]$ induces an algebra isomorphism 
$$f\SerrZ f\cong\SenrZ.$$
Thus, for $n<r$, the evaluation map is in fact an $\SenrZ$-$\sH_\sZ(r)$-bimodule isomorphism
\begin{equation}\label{bimod}
\ev: f\SerrZ e_{\emptyset}\lra \sT_\sZ(n,r),\;\;[A_\bsi]\lmt [A_\bsi](1).
\end{equation}
We record the right $\sH_\sZ(r)$-action from \eqref{ev} and \eqref{bimod} as follows:
\begin{equation}\label{H action}
[A_\bsi] T_{s_j}=\begin{cases}\up [A_{\bsi s_j}],&\text{ if }\ell(d_\bsi s_j)>\ell(d_\bsi), d_\bsi s_j\in\sD_{\la_\bsi} \text{ or }i_j<i_{j+1};\\
\up^2[A_\bsi],&\text{ if }\ell(d_\bsi s_j)>\ell(d_\bsi),d_\bsi s_j\not\in\sD_{\la_\bsi} \text{ or }i_j=i_{j+1};\\
(\up^2-1)[A_\bsi]+\up[A_{\bsi s_j}],&\text{ if }\ell(d_\bsi s_j)<\ell(d_\bsi), \text{ or }i_j>i_{j+1}.
\end{cases}
\end{equation}
Here, for $j=r$, $i_{r+1}=2n+1-i_r$ (cf. \eqref{hat i}).\footnote{The place permutation $\bsi s_j$ is ``truncated'' from the place permutation of $\fS_{2r}$ on all $\widehat\bsi$ defined in
\eqref{hat i} 
via the embedding $\sigma$ in \eqref{sigma}.}

Let $\bsH(r)=\sH_\sZ(r)\otimes \Q(\up)$ and $\bsS^\imath(n,r)=\SenrZ\otimes \Q(\up)$.
In both cases, we obtain the vector space isomorphisms
\begin{equation}\label{eta}
\eta_{r}=\eta_{n,r}:\begin{cases}
\Om_{2n}^{\otimes r}\lra  \bsS^\imath(n,r)e_\emptyset,\;\;\om_\bsi\lmt[A_\bsi],&\text{ if }n\geq r;\\
\Om_{2n}^{\otimes r}\lra  f\bsS^\imath(r,r)e_\emptyset,\;\;\om_\bsi\lmt[A_\bsi^\circ],&\text{ if }n< r.
\end{cases}
\end{equation}

The $\bsH(r)$-action defined in \eqref{H action} is transferred to a right $\bsH(r)$-action on $\Om_{2n}^{\otimes r}$ so that $\eta_{r}$ is an $\bsH(r)$-module isomorphism. Note that the transferred action here coincides with those given in \cite[(6.3),(6.4)]{BKLW}.

\begin{remark}\label{Sjnr} 
(1) If $\la\in\La(n+1,r)$, then $ S\backslash\{s_{\la_1+\cdots+\la_i}\mid i\in[1,n]\}$ also generates a parabolic subgroup $W_\la$ of $W$ that is not necessarily a subgroup of $W'$. Using this $W_\la$, define $x_\la$ as in \eqref{Tnr}. The endomorphism algebra
$$\SnrZ:=\End_{\sH_\sZ(r)}\bigg(\bigoplus_{\la\in\Lambda(n+1,r)}x_\la\sH_\sZ(r)\bigg)$$
is called the $q$-Schur algebra of type $B$. This algebra has a basis $\{[A]\mid A\in\Xi_{2n+1,2r+1}\}$ indexed by the matrix set $\Xi_{2n+1,2r+1}$. 

As observed in \cite[\S5]{BKLW}, this algebra contains a centraliser subalgebra $e\SnrZ e$, where $e^2=e\in\SnrZ$, that is isomorphic to $\SenrZ$; see \S4 for more details.

(2) There is a quantum coordinate algebra approach to both $\SnrZ$ and $\SenrZ$ developed by Lai, Nakano and Xiang in \cite{LNX} where they realise such an algebra as the dual of the $r$-th homogeneous component of the ``coordinate algebra'' of the corresponding $i$-quantum group.
\end{remark}

\section{Short multiplication formulas in $\SenrZ$}
We first embed $\SenrZ$ into $\SnrZ$ as a centralizer subalgebra of the form $e\SnrZ e$ for an idempotent $e$ and then derive some ``short'' multiplication formulas via their counterpart in $e\SnrZ e$ extracted from \cite{BKLW}. As a byproduct, we prove that the image of the map $\rho_{r}^\imath$ in \eqref{rhoi} is isomorphic to a subalgebra of $\SenrQ:=\SenrZ\otimes {\mathbb Q(\up)}$.

We start with the canonical embedding of $\fS_{2r}$ into $\fS_{2r+1}$. Let
$\iota:[1,2r]\to[1,2r+1]$ be the embedding defined by
$\iota(x)=\begin{cases}x,&\text{ if }x\leq r;\\x+1,&\text{ if }x>r.\end{cases}$
Then, $\iota$ induces an injective group homomorphism
$$\iota:\fS_{2r}\lra\fS_{2r+1},\;w\longmapsto \iota(w),$$
where $\iota(w)$ is the permutation that fixes $r+1$ and equals $\iota \circ w\circ\iota^{-1}$ on $[1,2r+1]\backslash\{r+1\}$.
In other words, we may identify $\fS_{2r}$ as a subgroup $\iota(\fS_{2r})$ of $\fS_{2r+1}$ consisting of permutations that fix $r+1$. 

\begin{remark}\label{BvsC}
 If $w_0$ is the longest element of $\fS_{2r+1}$ sending $i$ to $2r+2-i$, then $w_0(r+1)=r+1$ and so $w_0\in\iota(\fS_{2r})$. Thus, $w_0$ induces an (inner) automorphism $\tilde\theta$ sending $w$ to $w_0ww_0$ on $\fS_{2r+1}$ which restricts to the automorphism $\theta=\iota^{-1}\circ\tilde\theta\circ\iota$ on $\fS_{2r}$. Thus, we have the type $B$ identification $W=(\fS_{2r+1})^{\tilde\theta}$. Compare the type $C$ identification $W=(\fS_{2r})^\theta$ given in \eqref{sigma}.
\end{remark}


Now consider the following embeddings
\begin{equation}\label{Adag}
\aligned
(\;\;)^\dagger&:\Xien\lra \Xin,\;\;A=\begin{pmatrix}X&Y\\ Y'&X'\end{pmatrix}
\longmapsto A^\dagger=\begin{pmatrix}X&|&Y\\ \text{---}&1&\text{---}\\
Y'&|&X'\end{pmatrix},\\
(\;\;)^\dagger&:\La(n,r)\lra\La(2n+1,2r+1),\;\la\longmapsto
\hla^\dagger:=(\l_1,\l_2,\ldots,\l_n,1,\l_n,\ldots,\l_2,\l_1).
\endaligned
\end{equation}
where --- and $|$ represent a zero row and column, respectively.

 Let $\Xini$ be the image of $\Xi_{2n}$ in $\Xin$. Then $\Xini$ consists of matrices such that all entries in the $n+1$st row or the $n+1$st column are 0 except the $(n+1,n+1)$ entry which is 1. Let
 $$\Xinri=\Xini\cap\Xinr,\;\;\Xinli=\Xini\cap\Xinl.$$
 Then, $\Xinri=\{A^\dagger\mid A\in\Xienr\}$.
 
We also set 
$$\hLa^\dagger(n,r)=\{\hla^\dagger\mid \la\in\La(n,r)\}\subset \La(2n+1,2r+1).$$
Clearly, $\hLa^\dagger=\{\ro(A)\mid A\in \Xinri\}=\{\wla\mid\la\in\La(n+1,r),\la_{n+1}=0\}$ under the notation of \cite[(3.0.1)]{DW} (or Remark \ref{Sjnr} above). 


Let $e=\sum_{\la\in\La(n,r)}[\diag(\hla^\dagger)]$. Then $e$ is an idempotent in $\SnrZ$. Define the centraliser subalgebra
\begin{equation}\label{Sji}
\SnriZ:=e\SnrZ e.
\end{equation}
\begin{lemm}
 There is an algebra embedding $\fkf:\SenrZ\lra\SnrZ$ sending $[A]$ to $[A^\dagger],$
which induces an algebra isomorphism 
\begin{equation}\label{fkf}
\fkf:\SenrZ\lra\SnriZ;\;\;[A]\longmapsto[A^\dagger].
\end{equation} 
\end{lemm}


For $i,j\in[1,2n]$, let $\bfe_i=(0,\ldots,0,\underset{(i)}1,0,\ldots0)\in \Z^{2n}$ and
$E_{i,j}\in \text{Mat}_{2n}(\mathbb{N})$ the standard matrix units. Define 
\begin{equation}\label{Eij}
\bfe^\theta_i=\bfe_i+\bfe_{2n+1-i},\quad E^{\theta}_{i,j}:=E_{i,j}+E_{2n+1-i,2n+1-j}\in\Xi_{2n}.
\end{equation} 
Then  $\bfe_{2n+1-i}^\theta=\bfe_i^\theta=\ro(E^\theta_{i,j})=\co(E^\theta_{j,i})$ for all $i,j\in[1,2n]$ and $E^{\theta}_{i,j}=E^\theta_{2n+1-i,2n+1-j}$. In particular, $E_{n,n+1}^\theta=E_{n+1,n}^\theta$.

Similarly, for the standard basis elements $\bfe'_i\in \Z^{2n+1}$ and
$E'_{i,j}\in \text{Mat}_{2n+1}(\mathbb{N})$, define
\begin{equation}\label{E'ij}
\bfe^{\prime\theta}_i=\bfe'_i+\bfe'_{2n+2-i},\quad E^{\prime\theta}_{i,j}:=E'_{i,j}+E'_{2n+2-i,2n+2-j}\in\Xi_{2n+1}
\end{equation}

For $A \in \Xi_{2n}$, $h\,\in\,[1,n]$ and $p\in[1,2n]$, let
\begin{equation}\label{beta}
\aligned
\b_p(A,h)&=\sum_{j\geq p}a_{h,j}-\sum_{j> p }a_{h+1,j},\\
\b'_p(A,h)&=\sum_{j\leq p}a_{h+1,j}-\sum_{j< p}a_{h,j}.
\endaligned
\end{equation}
Note that $\b_p(A,h)$ is slightly different from the one defined in \cite[(4.0.1)]{DW} for matrices in $\Xi_{2n+1}$. Moreover, we have the symmetry property\footnote{There is no such a symmetry property for the similarly named functions in \cite[(4.0.1)]{DW}.}   $\b'_p(A,n)=\b_{2n+1-p}(A,n)$.

As set in \cite[\S5.1]{BKLW}, the multiplication formulas given in \cite[Lem.~3.2]{BKLW} or \cite[Lem. 4.1]{DW} continue to hold in $\SnriZ$ whenever $h\neq n$. An extra formula related to the generator $t$, replacing $e_n,f_n$, is given in \cite[Lem.~ A.13]{BKLW}. We now write these formulas in $\SenrZ$.

\begin{lemm} \label{keyMF}
The $\sZ$-algebra $\SenrZ$ has a basis $\{[A]\mid A\in\Xienr\}$. If $A=(a_{i,j})\in\Xienr$, $\l\in\La(n,r-1)$ and $1\leq h<n$, the following multiplication formulas hold in $\SenrZ$:
\begin{enumerate}
\item[(1)] $[E^\theta_{h,h+1}+\hla]\cdot[A]=\delta_{\bfe_{h+1}^\theta+\hla,\ro(A)}\displaystyle\sum_{p\in[1,2n]\atop a_{h+1,p}\geq1}\up^{\beta_p(A,h)}\overline{[\![a_{h,p}+1]\!]}[A+E^{\theta}_{h,p}-E^{\theta}_{h+1,p}]$
;
\item[(2)] $[E^\theta_{h+1,h}+\hla]\cdot[A]=\delta_{\bfe_{h}^\theta+\hla,\ro(A)}\displaystyle\sum_{p\in[1,2n]\atop a_{h,p}\geq 1}\up^{\b'_p(A,h)}\overline{[\![a_{h+1,p}+1]\!]}[A-E^{\theta}_{h,p}+E^{\theta}_{h+1,p}]$;

\item[(3)] $
[E^{\theta}_{n+1,n}+\hla]\cdot[A]=\delta_{\bfe^\theta_n+\hla,\ro(A)}\Big(c_A[A]+\!\!\displaystyle\sum_{p\in[1, 2n]\atop a_{n,p}\geq1} \up^{\b'_p(A,n)-\epsilon}\overline{[\![a_{n+1,p}+1]\!]}[A-E^{\theta}_{n,p}+E^{\theta}_{n+1,p}]\Big)$,
where $c_A$ and $\epsilon=\delta^\leq_{n+1,p}$ are defined by
\begin{equation}\label{c_A}
c_A=\up^{-\sum_{j\leq n}a_{n,j}}\big(\up^{\sum_{j\leq n}a_{n+1,j}}-\up^{-\sum_{j\leq n}a_{n+1,j}}\big) 
\text{ and }\delta^{\leq}_{i,j}=\begin{cases}1,&\text{ if }i\leq j;\\0,&\text{ if }i>j.\end{cases}
\end{equation}
 \end{enumerate}
\end{lemm}
\begin{proof}If we choose $A:=A^\dagger$, $\wla:=\hla^\dagger$ in \cite[Lem. 4.1]{DW} for some $A\in\Xienr,\la\in\La(n,r)$, then both formulas in  \cite[Lem. 4.1(1)\&(2)]{DW} are closed in $\SnriZ$ for $h\in[1,n)$.
So,  (1) and (2) are the $\fkf$-inverse images of the two. 

To see (3), we use $A^\dagger=(a_{i,j}^\dagger)$ to replace $A$ in the displayed formula in \cite[Lem.~ A.13]{BKLW}, which holds in $\SnriZ$.
Thus, by using the notation in \eqref{E'ij}, each term on the RHS of the formula has the form $[A^\dagger-E^{\prime\theta}_{n,p'}+E^{\prime\theta}_{n+2,p'}]=[(A-E^{\theta}_{n,p}+E^{\theta}_{n+1,p})^\dagger]$ if $p'\neq n+1$ ($p=p'$ for $p'\leq n$ and $p=p'-1$ for $p'\geq n+2$) or $[A^\dagger]$ if $p'=n+1$ which has  coefficient
$$\up^{\sum_{j\leq p'}a_{n+2,j}^\dagger-\sum_{j<p'}a_{n,j}^\dagger-\sum_{j>n+1}\delta_{p',j}}
\overline{[\![a_{n+2,p'}^\dagger+1]\!]}=\begin{cases}
\up^{\b'_p(A,n)-\delta^\leq_{n+1,p}}\overline{[\![a_{n+1,p}+1]\!]},&\text{ if }p'\neq n+1;\\
\up^{\sum_{j\leq n}a_{n+1,j}-\sum_{j\leq n}a_{n,j}},&\text{ if }p'=n+1.
\end{cases}$$
Here, in the $p'=n+1$ case, we used the fact that column $n+1$ or row $n+1$ of $A^\dagger$ has zero entries except the $(n+1,n+1)$ entry.

On the other hand, by \cite[(A.9)]{BKLW}, the left hand side of  the formula in \cite[Lem.~ A.13]{BKLW} contains a summand $\up^{-\ro(A^\dagger)_n}[A^\dagger]$ which is moved to the right hand side. Thus, the coefficient of $[A^\dagger]$ is
$\up^{\sum_{j\leq n}a_{n+1,j}-\sum_{j\leq n}a_{n,j}}-\up^{-\sum_{i\in[1,2n]}a_{n,i}},$
which equals $c_A$ since $\sum_{n+1\leq i}a_{n,i}=\sum_{j\leq n}a_{n+1,j}$. Hence, (3) follows.
\end{proof}
The following values will be used in the sequel: for $h\in[1,n]$, $a\in\N$,
\begin{equation}\label{cE}
c_{aE^{\theta}_{h,h+1}}=c_{aE^{\theta}_{h+1,h}}=\begin{cases}0,&\text{ if }h\neq n;\\\up^a-\up^{-a},&\text{ if }h=n.\end{cases}
\end{equation}

For $A=(a_{i,j})\in\Xienr$ and $\nu=(\nu_i)\in\La(2n,m)$ ($m>0$), we set
\begin{equation}\label{order yi}
\nu\leq \row_h(A)\iff \nu_i\leq a_{h,i}\;\,\forall i\in[1,2n], \text{ where }\row_{h}(A)=(a_{h,1},\ldots,a_{h,2n}).
\end{equation}
When $h<n$, multiplication formulas in \cite[Thm.~3.7(1)\&(2)]{BKLW} are closed in $\SnriZ$. For later use, we record their counterpart in $\SenrZ$ as follows.

\begin{prop} \label{keyMFm} If $A=(a_{i,j})\in\Xienr$, $m>0$, $\l\in\La(n,r-m)$, and $1\leq h<n$, the following multiplication formulas hold in $\SenrZ$:
\begin{enumerate}
\item[(1)] $[mE^\theta_{h,h+1}+\hla][A]=\varepsilon\displaystyle\sum_{\nu\in\La(2n,m)\atop \nu\leq \row_{h+1}(A)}\up^{\beta_\nu(A,h)}\prod_{u=1}^{2n}\overline{\left[\!\!\left[a_{h,u}+\nu_u\atop \nu_u\right]\!\!\right]}[A+\sum_{u=1}^{2n}\nu_u(E^{\theta}_{h,u}-E^{\theta}_{h+1,u})]$,
where $\varepsilon=\delta_{m\bfe_{h+1}^\theta+\hla,\ro(A)}$ and $\b_\nu(A,h)=\sum_{j\geq p}a_{h,j}\nu_p-\sum_{j> p }a_{h+1,j}\nu_p+\sum_{j<p}\nu_j\nu_p$.
\item[(2)] $[mE^\theta_{h+1,h}+\hla][A]=\varepsilon'\displaystyle\sum_{\nu\in\La(2n,m)\atop \nu\leq \row_{h}(A)}\up^{\beta_\nu'(A,h)}\prod_{u=1}^{2n}\overline{\left[\!\!\left[a_{h+1,u}+\nu_u\atop \nu_u\right]\!\!\right]}[A+\sum_{u=1}^{2n}\nu_u(E^{\theta}_{h+1,u}-E^{\theta}_{h,u})]$,
where $\varepsilon'=\delta_{m\bfe_{h}^\theta+\hla,\ro(A)}$ and $\b'_\nu(A,h)=\sum_{j\leq p}a_{h+1,j}\nu_p-\sum_{j< p}a_{h,j}\nu_p+\sum_{j>p}\nu_j\nu_p.$
 \end{enumerate}
\end{prop}
It seems too complicated to write down a formula for the product $[mE^\theta_{n+1,n}+\hla][A]$.

We now give an application of the multiplication formulas given in Lemma \ref{keyMF}. 
The following result allows us to transfer the $\SenrZ$-$\sH_\sZ(r)$-bimodule structure on $\sT_\sZ(n,r)$, after base change, to a $\bfUin$-$\bsH(r)$-bimodule structure on $\Omega_{2n}^{\otimes r}$ via \eqref{ev}, \eqref{bimod}, and \eqref{eta}.

\begin{theo}\label{rH action} The right $\bsH_\sZ(r)$-module structure on $\Om_{2n}^{\otimes r}$ defined by \eqref{H action} and \eqref{eta} commutes with the action of $\bfUin$.
In other words, the map $\rho_{r}^\imath:\bfUin\lra\End(\Om_{2n}^{\otimes r})$ defined in \eqref{rhoi} can be 
refined to 
\begin{equation}\label{rho3}
\rho_{r}^\imath=\rho_{n,r}^\imath:\bfUin\lra\End_{\bsH(r)}(\Om_{2n}^{\otimes r}).
\end{equation}
\end{theo}
\begin{proof}Recall the $\bsH(r)$-action on $\Om_{2n}^{\otimes r}$ via \eqref{H action} and the $\bsH(r)$-module isomorphisms $\eta_{r}$ in  \eqref{eta}.
By restriction to the subalgebra $\bsH(\fS_r)$, $\Om_{2n}^{\otimes r}$ becomes a right $\bsH(\fS_r)$-module
(cf. \cite[(14.6.4)]{DDPW08}) and the map $\rho_{r}$ in \eqref{rho} induces an algebra homomorphism, denoted by $\rho_{r}$ again,
\begin{equation}\label{rho2}
\rho_{r}:\bfU(\mathfrak{gl}_{2n})\lra\End_{\bsH(\fS_r)}(\Om_{2n}^{\otimes r}).
\end{equation}
We now prove that the restriction map $\rho^\imath_{r}$ sends $\bfUin$ into $\End_{\bsH(r)}(\Om_{2n}^{\otimes r})$.

We first assume $n\geq r$. In this case, $\eta_{r}$ induces an algebra isomorphism 
\begin{equation}\label{teta}
\tilde\eta_{r}:=\tilde\eta_{n,r}:\End_{\bsH(r)}(\Om_{2n}^{\otimes r})\overset\sim\lra\bsS^\imath(n,r).
\end{equation}
We claim that, under the linear isomorphism $\eta_{r}$ in \eqref{eta}, the action formulas on the basis $\{\omega_\bsi\mid\bsi\in I(2n,r)\}$ given in Proposition \ref{U action}(1)--(4) coincide with the action formulas of certain elements in $\bsS^\imath(n,r)$ on the basis $\{[A_\bsi]\mid\bsi\in I(2n,r)\}$. More precisely, we claim that, for all $j\in[1,n],h\in[1,n)$, and $\bsi\in I(2n,r)$ with $\wtt(\bsi)=\la$ (so that $\hla=\ro(A_\bsi)$),
\begin{itemize}
\item[(1)] $\eta_{r}(\rho_{r}^\imath(d_j).\om_\bsi)=\up^{-\la_j}[\hla]\cdot[A_\bsi]$;
\item[(2)] $\eta_{r}(\rho_{r}^\imath(e_h).\om_\bsi)=[E_{h+1,h}^\theta+\hla-\bfe_h^\theta]\cdot[A_\bsi]$;
\item[(3)] $\eta_{r}(\rho_{r}^\imath(f_h).\om_\bsi)=[E_{h,h+1}^\theta+\hla-\bfe_{h+1}^\theta]\cdot[A_\bsi]$;
\item[(4)] $\eta_{r}(\rho_{r}^\imath(t).\om_\bsi)=([E_{n+1,n}^\theta+\hla-\bfe_n^\theta]+\up^{-\la_n}[\hla])\cdot[A_\bsi]$.
\end{itemize}
Note that, if $\la_h=\ro(A_\bsi)_h=0$ in (2), or $\la_{h+1}=\ro(A_\bsi)_{h+1}=0$ in (3), then both sides are zeros since there are no components in $\bsi$ equal $h$ or $h+1$ in these cases.

By the claim, we see from \eqref{eta} that the $\bfUin$ action on $\Omega_{2n}^{\otimes r}$ commutes with the $\bsH(r)$ action transferred above. Hence, we have $\im(\rho^\imath_{r})\subseteq\End_{\boldsymbol{\bsH}(r)}(\Om_{2n}^{\otimes r})$ in this case.


We now prove (1)--(4) in the claim. Recall the definition of $A_\bsi$ in  \eqref{Ai}.

For
Part (1), it suffices to prove $\ro(A_\bsi)_j=\delta(j,\bsi)$. This is clear since
$$\aligned
\ro(A_\bsi)_j&=|\{l\mid l\in[1,r],a_{j,l}=1=\delta_{j,i_l}\}\cup\{l\mid l\in[2n-r+1,2n],a_{j,l}=1\}|\\
&=|\{l\mid l\in[1,r],i_l=j\}\cup\{ l\mid l\in[2n-r+1,2n],i_{2n+1-l}=2n+1-j\}|\\
&=|\{k\mid 1\leq k\leq r,i_k=j\}|+|\{k\mid 1\leq k\leq r,i_{k}=2n+1-j\}|=\delta(j,\bsi).
\endaligned
$$
Parts (2) and (3) can be easily checked by the short multiplication formulas in Lemma \ref{keyMF}(1)\&(2) by mimicking part of the proof of \cite[Thm.~7.1]{DW}.

Finally, we prove Part (4). By Lemma \ref{keyMF}(3), we have
$$\aligned
[E_{n+1,n}^\theta+\ro(A_\bsi)-\bfe_n^\theta]\cdot[A_\bsi]&=c_{A_\bsi}[A_\bsi]+\sum_{l\in[1,2n],a_{n,l}\geq 1}\up^{\b'_l(A_\bsi,n)-\delta^{\leq}_{n+1,l}}[A_\bsi-E^{\theta}_{n,l}+E^{\theta}_{n+1,l}]\\
&=c_{A_\bsi}[A_\bsi]+\sum_{\substack{l\in[1,r]\\a_{n,l}=\delta_{n,i_l}=1}}\up^{\b'_l(A_\bsi,n)}[A_\bsi-E^{\theta}_{n,l}+E^{\theta}_{n+1,l}]\\
&\quad\ +\ \sum_{\substack{l\in[2n-r+1,2n]\\a_{n,l}=\delta_{n+1,i_{2n+1-l}}=1}}\up^{\b'_l(A_\bsi,n)-1}[A_\bsi-E^{\theta}_{n,l}+E^{\theta}_{n+1,l}]\\
&=c_{A_\bsi}[A_\bsi]+\sum_{\substack{l\in[1,r]\\i_l=n}}\up^{\b'_l(A_\bsi,n)}[A_\bsi-E^{\theta}_{n,l}+E^{\theta}_{n+1,l}]\\
&\quad\ +\ \sum_{\substack{l\in[1,r]\\ i_{l}=n+1}}\up^{\b'_{2n+1-l}(A_\bsi,n)-1}[A_\bsi-E^{\theta}_{n,2n+1-l}+E^{\theta}_{n+1,2n+1-l}].
\endaligned
$$
We now compare this action with the action formula in Proposition \ref{U action}(4). First, we have
$$\aligned
\eta_r(\omega_{i_1}\cdots\omega_{i_{l-1}}\omega_{n+1}\omega_{i_{l+1}}\cdots\omega_{i_r})&=[A_\bsi-E^{\theta}_{n,l}+E^{\theta}_{n+1,l}],\;\;\text{ and }\;\\
\eta_r(\omega_{i_1}\cdots\omega_{i_{l-1}}\omega_{n}\omega_{i_{l+1}}\cdots\omega_{i_r})&=[A_\bsi-E^{\theta}_{n+1,l}+E^{\theta}_{n,l}]=[A_\bsi-E^{\theta}_{n,2n+1-l}+E^{\theta}_{n+1,2n+1-l}].\endaligned$$
It remains to verify that the corresponding coefficients are equal, i.e., to prove that
\begin{eqnarray*}
\begin{cases}
({\rm a})\;\;c_{A_\bsi}+\up^{-ro(A_{\bsi})_n}=\up^{\tau_0};&\\
({\rm b})\;\;\b'_l(A_\bsi,n)=\tau_1(l), &\text{ if }i_l=n;\\
({\rm c})\;\;\b'_{2n+1-l}(A_\bsi,n)=\tau_1(l)+1 &\text{ if }i_l=n+1.
\end{cases}
\end{eqnarray*} 
Since $c_{A_\bsi}=\up^{-\sum_{j\leq n}a_{n,j}}\big(\up^{\sum_{j\leq n}a_{n+1,j}}-\up^{-\sum_{j\leq n}a_{n+1,j}}\big)=\up^{\sum_{j\leq n}a_{n+1,j}-\sum_{j\leq n}a_{n,j}}-\up^{-ro(A_{\bsi})_n}$, it follows that
$c_{A_\bsi}+\up^{-ro(A_{\bsi})_n}=\up^{\sum_{j\leq n}a_{n+1,j}-\sum_{j\leq n}a_{n,j}}=\up^{|\{k\mid 1\leq k\leq r,i_k=n+1\}|-|\{k\mid 1\leq k\leq r,i_k=n\}|}=\up^{\tau_0},$ proving (a). 

For (b), this is the $i_l=n$ case:
\begin{eqnarray*}
&&\b'_l(A_\bsi,n)=\sum_{k\leq l}a_{n+1,k}-\sum_{k<l}a_{n,k}\\
&=&|\{k\mid 1\leq k\leq l,a_{n+1,k}=\delta_{n+1,i_k}=1\}|-|\{k\mid 1\leq k< l,a_{n,k}=\delta_{n,i_k}=1\}|\\
&=&|\{k\mid 1\leq k< l,i_k=n+1\}|-|\{k\mid 1\leq k< l,i_k=n\}|=\tau_1(l).
\end{eqnarray*}
Finally, for the $i_l=n+1$ case, we have
\begin{eqnarray*}
&&\b'_{2n+1-l}(A_\bsi,n)=\sum_{k\leq 2n+1-l}a_{n+1,k}-\sum_{k<2n+1-l}a_{n,k}\\
&=&|\{k\mid 1\leq k\leq r,a_{n+1,k}=1\}|+|\{k\mid 2n+1-r\leq k\leq 2n+1-l,a_{n+1,k}=1=a_{n,2n+1-k}\}|\\
&&-\ |\{k\mid 1\leq k\leq r,a_{n,k}=1\}|-|\{k\mid 2n+1-r\leq k< 2n+1-l,a_{n,k}=1=a_{n+1,2n+1-k}\}|\\
&=&|\{k\mid 1\leq k\leq r,i_k=n+1\}|+|\{k\mid l\leq k\leq r,i_k=n\}|\\
&&-\ |\{k\mid 1\leq k\leq r,i_k=n\}|-|\{k\mid l<k\leq r, i_k=n+1\}|\\
&=&|\{k\mid 1\leq k\leq l,i_k=n+1\}|-\ |\{k\mid 1\leq k< l,i_k=n\}|\\
&=&|\{k\mid 1\leq k< l,i_k=n+1\}|-\ |\{k\mid 1\leq k< l,i_k=n\}|+1=\tau_1(l)+1,
\end{eqnarray*}
proving (c) and hence (4). This completes the proof for the $n\geq r$ case.

Assume now $n<r$. We have proved that $\rho_{r,r}^\imath(\bfU^\imath(r))\subseteq\End_{\bsH(r)}(\Om_{2r}^{\otimes r})$. Consider the algebra embedding 
$\iota_1:\bfU^\imath(n)\lra \bfU^\imath(r)$ (resp., the space embedding $\Om_{2n}\to \Om_{2r}$) induced by the index embedding 
$$[1,n]\lra[1,r]\;(\text{resp., } [1,2n]\lra[1,2r]),\;i\lmt r-n+i,\forall i.$$
The latter induces an $\bsH(r)$-module embedding $\Om_{2n}^{\otimes r}\to\Om_{2r}^{\otimes r}$ so that its image is a direct summand of  the $\bsH(r)$-module $\Om_{2r}^{\otimes r}$. Thus, we have a centraliser subalgebra embedding
$$\iota_2:  \End_{\bsH(r)}(\Om_{2n}^{\otimes r})\lra\End_{\bsH(r)}(\Om_{2r}^{\otimes r})$$
whose image $\im(\iota_2)=(\tilde\eta_{r,r})^{-1}(f\bsS^\imath(r,r)f)$ (see \eqref{teta}), where $f=\sum_{\la\in\La(n,r)}[\widehat{\la^\circ}]$.
Now, one sees easily the inclusion  $\rho_{r,r}^\imath \iota_1(\bfU^\imath(n))\subseteq \iota_2\big( \End_{\bsH(r)}(\Om_{2n}^{\otimes r})\big)$. Hence,
$\rho_{r}^\imath(\bfU^\imath(n))\subseteq\End_{\bsH(r)}(\Om_{2n}^{\otimes r})$ where $\rho_{r}^\imath=\iota_2^{-1}\rho_{r,r}^\imath \iota_1$.
\end{proof}

We will see in \S6 that the map $\rho^\imath_r$ given in \eqref{rho3} is surjective.

In the next result, we identify $\End_{\bsH(r)}(\Omega^{\otimes r}_{2n})$ with
 $\bsS^\imath(n,r)$ via \eqref{teta}, For $j\in[1,n]$ and $A\in\{E_{j,j+1}^\theta,E^\theta_{j+1,j}\}$, let
 \begin{equation}\label{FirstAjr}
 O(-\bfe_j,r)=\sum_{\la\in\Lambda(n,r)}\up^{-\la_j}[\hla], \quad A(\bfl,r)=\sum_{\mu\in\La(n,r-1)}[A+\hmu],
\end{equation}

\begin{coro}\label{motivation ex}
The $\mathbb Q(\up)$-algebra homomorphism $\tilde\eta_r\circ\rho_{r}^\imath:\bfUin\to\End_{\bsH(r)}(\Omega^{\otimes r}_{2n})\overset\sim\to\SenrQ$ has the following images on generators: for all $h\in[1,n),j\in[1,n]$, 
$$d_j\mapsto O(-\bfe_j,r),\;e_h\mapsto  E^\theta_{h+1,h}(\bfl,r),\;f_h\mapsto  E^\theta_{h,h+1}(\bfl,r), \; t\mapsto E^\theta_{n+1,n}(\bfl,r)+O(-\bfe_n,r).$$
\end{coro}
\begin{proof}This follows easily from the equations (1)--(4) in the proof of the theorem since the right hand sides of (1)--(4) are equal to $O(-\bfe_j,r)[A_\bsi]$, $E^\theta_{h+1,h}(\bfl,r)[A_\bsi]$, $E^\theta_{h,h+1}(\bfl,r)[A_\bsi]$, and $\big(E^\theta_{n+1,n}(\bfl,r)+O(-\bfe_n,r)\big)[A_\bsi]$, respectively.
\end{proof}

\section{A triangular relation in $\SenrZ$}
In this section, we develop a triangular relation for two bases of $\SenrZ$. This is built on the determination of leading terms in certain multiplication formulas with respect to a preorder.
Recall the preorder $\preccurlyeq$ in \cite[(3.22)]{BKLW} defined on $\Xi_N$ ($N=2n$ or $2n+1$): for $A=(a_{i,j}),B=(b_{i,j})\in\Xi_N$,
\begin{equation}\label{precN}
A\preccurlyeq B \Longleftrightarrow
\sum_{i\leq u;j\geq v}a_{i,j}\leq \sum_{i\leq u;j\geq v}b_{i,j}, \mbox{ for all $1\leq u<v\leq N$}.
\end{equation}
Equivalently, for $N=2n$, we have
\begin{equation}\label{prec}
A\preccurlyeq B\Longleftrightarrow\begin{cases}
\sum_{i\leq u;j\geq v}a_{i,j}\leq \sum_{i\leq u;j\geq v}b_{i,j}, &\mbox{ for all }u\in[1,n],v\in[1,2n],u<v;\\
\sum_{i\geq u;j\leq v}a_{i,j}\leq \sum_{i\geq u;j\leq v}b_{i,j},&\mbox{for all }u,v\in[1,n], u>v.
\end{cases}
\end{equation}
We write $A\prec B$ if $A\preccurlyeq B$ and $B\not\preccurlyeq A$. Note that, with \eqref{Adag}, we have $A\preccurlyeq B\iff A^\dagger\preccurlyeq B^\dagger$.

We first derive some multiplication formulas in $\SnrZ$ (or more precisely in its centraliser subalgebra $\SnriZ$).
With the notations $\bfe^{\prime\theta}_n, E^{\prime\theta}_{i,j}$ given in \eqref{E'ij} and by the display \cite[(5.4)]{BKLW}, the following formula holds in $\SnriZ$: for $m\geq0$, $\la\in\La(n,r)$, and $\beta(i):=i\l_{n}-{i+1\choose2}$, 
$$\aligned{}
[mE^{\prime\theta}_{n,n+1}+\hla^\dagger]&\cdot[m E^{\prime\theta}_{n+1,n}+\hla^\dagger]=\sum^m_{i=0}\up^{\b(i)}
\overline{\begin{bmatrix}\!\!\begin{bmatrix}\l_n+i\\i\end{bmatrix}\!\!\end{bmatrix}}[(m-i)E^{\prime\theta}_{n+2,n}+\hla^\dagger+i\bfe^{\prime\theta}_{n}]\\
&=[mE^{\prime\theta}_{n+2,n}+\hla^\dagger]+\sum^m_{i=1}\up^{-\frac{i(i+1)}{2}}
\begin{bmatrix}\l_n+i\\i\end{bmatrix}[(m-i)E^{\prime\theta}_{n+2,n}+\hla^\dagger+i\bfe^{\prime\theta}_{n}],
\endaligned$$
where we use (\ref{0.n-m}) to get the second equality. Thus, for $i\in[1,m]$, we have
$$\aligned{}
[(m&-i)E^{\prime\theta}_{n,n+1}+\hla^\dagger+i\bfe^{\prime\theta}_{n}]\cdot[(m-i)E^{\prime\theta}_{n+1,n}+\hla^\dagger+i\bfe^{\prime\theta}_{n}]=\\
&[(m-i)E^{\prime\theta}_{n+2,n}+\hla^\dagger+i\bfe^{\prime\theta}_{n}]+\sum^{m-i}_{j=1}\up^{-\frac{j(j+1)}2}
\begin{bmatrix}\l_n+i+j\\j\end{bmatrix}[(m-i-j)E^{\prime\theta}_{n+2,n}+\hla^\dagger+(i+j)\bfe^{\prime\theta}_{n}].
\endaligned$$
Hence, for $0\leq i\leq m$, there exist $p_i\in\sZ$ with $p_0=1$ such that
\begin{equation}\label{trump}
[mE^{\prime\theta}_{n+2,n}+\hla^\dagger]=\sum_{i=0}^mp_i\big([(m-i)E^{\prime\theta}_{n,n+1}+\hla^\dagger+i\bfe^{\theta}_{n}]\cdot[(m-i)E^{\prime\theta}_{n+1,n}+\hla^\dagger+i\bfe^{\theta}_{n}]\big).
\end{equation}
We first derive a transferable leading term formula in $\SnriZ$.
The ``lower terms'' in an expression of the form $[M]+(\text{lower terms})_\preccurlyeq$ represents a linear combination of $[B]$ with $B\prec M$.
\begin{lemm}\label{TLn}
Suppose that $A\in \Xinri$ satisfies that:  $a_{n,k}\geq m>0$ for some $k\in[1,n]$ and $a_{n,j}=0$, $\forall\ j<k$ and $j\geq 2n+2-k$.\footnote{The condition $a_{n,j}=0$, $\forall\ j\geq 2n+2-k$ is equivalent to $a_{n+2,j}=0$, $\forall j\leq k$.}
Then, for any $\wla=\ro(A)-m\bfe_n^{\prime\theta}$, we have in $\SnriZ$
\begin{equation}\label{trump1}[mE^{\prime\theta}_{n+2,n}+\wla]\cdot[A]=[A-mE^{\prime\theta}_{n,k}+mE^{\prime\theta}_{n+2,k}]+({\rm lower\ terms})_{\preccurlyeq}.
\end{equation}
\end{lemm}
\begin{proof}
Since $A-(m-i)E^{\prime\theta}_{n,k}+(m-i)E^{\prime\theta}_{n+2,k}\prec A-mE^{\prime\theta}_{n,k}+mE^{\prime\theta}_{n+2,k}$ for all $i\in[1,m]$, by \eqref{trump}, it suffices to prove that for every $m$ with $a_{n,k}\geq m\geq1$, the leading term in
the product $[mE^{\prime\theta}_{n,n+1}+\wla][mE^{\prime\theta}_{n+1,n}+\wla]\cdot[A]$ is $[A-mE^{\prime\theta}_{n,k}+mE^{\prime\theta}_{n+2,k}]$.

Let $N=2n+1$ and let $\bft^{(0)}=(t_i)=(t_1,t_2,\dots,t_{N})\in\La(N,m)$, where $t_k=m$ and $t_i=0$ for $i\neq k$. Let $\La(N,m)_0=\La(N,m)-\{ \bft^{(0)}\}$. By \cite[Thm.~ 3.7(b)]{BKLW}, we have
\begin{equation}\label{eq-A12}
\aligned{}
[m E^{\prime\theta}_{n+1,n}&+\wla]*[A]=[A-mE^{\prime\theta}_{n,k}+mE^{\prime\theta}_{n+1,k}]+\sum_{\substack{\bft\in\La(N,m)_0\\\bft\leq
\text{row}_n(A)}}f_\bft[A-\sum_{1\leq u\leq N}t_{u}(E^{\prime\theta}_{n,u}-E^{\prime\theta}_{n+1,u})],\endaligned
\end{equation}
where $f_\bft=\up^{\b''(\bft)}\prod_{j<n+1}\overline{\begin{bmatrix}\!\!\begin{bmatrix}t_j+t_{N+1-j}\\t_j\end{bmatrix}\!\!\end{bmatrix}}$ with $\b''(\bft)$ as defined in \cite[(3.21)]{BKLW}, and $\bft\leq
\text{row}_n(A)$ means $t_u\leq a_{n,u}$ for all $u$. 
Note that we have $t_{n+1}=0$, since $a_{n,n+1}=0$.
 (So, the missing factor $\prod_{i=0}^{t_{n+1}-1}\frac{\overline{[\![2i+2]\!]}}{\overline{[\![i+1]\!]}}$ in $f_t$ is 1.)
Note also that, since $a_{n,j}=0$, $\forall\ j<k$ and $j\geq N+1-k$, it follows that that $t_{u}=0$ for all $u<k$ and $u\geq N+1-k$.

Moreover, one checks easily by the definition that $A-mE^{\prime\theta}_{n,k}+mE^{\prime\theta}_{n+1,k}$ is the leading term on the right hand side of \eqref{eq-A12}.

Let $\bft^{(0')}=(t_i)=(t_1,t_2,\dots,t_{N})$, where $t_{N+1-k}=m$ and $t_i=0$ for $i\neq N+1-k$.
By \cite[Thm.~3.7(a)]{BKLW}, we have, for $A'=A-mE^{\prime\theta}_{n,k}+mE^{\prime\theta}_{n+1,k}=(a'_{ij})$, $\La(N,m)_{0'}:=\La(N,m)-\{ \bft^{(0')}\}$ and some $g_{\bft'}\in\sZ$,
\begin{equation*}\aligned
&[m E^{\prime\theta}_{n,n+1}+\wla]\cdot[A']=[mE^{\prime\theta}_{n,n+1}+\wla]\cdot[A-mE^{\prime\theta}_{n+2,N+1-k}+mE^{\prime\theta}_{n+1,N+1-k}]\\
&=[A-mE^{\prime\theta}_{n+2,N+1-k}+mE^{\prime\theta}_{n+1,N+1-k}+mE^{\prime\theta}_{n,N+1-k}-mE^{\prime\theta}_{n+1,N+1-k}]\\
&\qquad\qquad\qquad+
\sum_{\bft'\in\La(N,m)_{0'}\atop \bft'+(\bft')^\tau\leq \text{row}_{n+1}(A')}g_{\bft'}[A-mE^{\prime\theta}_{n,k}+mE^{\prime\theta}_{n+1,k}+\sum_u{t'_{u}}(E^{\prime\theta}_{n,\mu}-E^{\prime\theta}_{n+1,u})]\\
&=[A-mE^{\prime\theta}_{n,k}+mE^{\prime\theta}_{n+2,k}]+\!\!
\sum_{\bft'\in\La(N,m)_{0'}\atop \bft'+(\bft')^{\tau}\leq \text{row}_{n+1}(A')}g_{\bft'}[A-mE^{\prime\theta}_{n,k}+mE^{\prime\theta}_{n+1,k}+\sum{t'_{\mu}}(E^{\prime\theta}_{n,\mu}-E^{\prime\theta}_{n+1,\mu})],\\
\endaligned
\end{equation*}
where 
$\bft'+(\bft')^\tau=(t'_1+t'_N,t'_2+t'_{N-1},\ldots,t'_N+t'_1)$, and $\bft'+(\bft')^\tau\leq \text{row}_{n+1}(A')$ means $t'_{u}+t'_{N+1-u}\leq a'_{n+1,u}$ for all $u$.
 
Since $a'_{n+1,u}=0$ for all $u\neq k,N+1-k,n+1$, it follows that
$  t'_{u}=0\ {\rm for \ all\ }\mu\neq k,N+1-k.  $
Thus, $m=t'_{k}+t'_{N+1-k}\leq a'_{n+1,k}=m$ and
\begin{eqnarray*}\label{tkk}
&&A-mE^{\prime\theta}_{n,k}+mE^{\prime\theta}_{n+1,k}+\sum_u{t'_{u}}(E^{\prime\theta}_{n,u}-E^{\prime\theta}_{n+1,u})\\
&=&A-mE^{\prime\theta}_{n,k}+mE^{\prime\theta}_{n+1,k}+t'_{k}(E^{\prime\theta}_{n,k}-E^{\prime\theta}_{n+1,k})+{t'_{N+1-k}}(E^{\prime\theta}_{n,N+1-k}-E^{\prime\theta}_{n+1,N+1-k})\nonumber\\
&=&A-mE^{\prime\theta}_{n,k}+mE^{\prime\theta}_{n+1,k}+t'_{k}E^{\prime\theta}_{n,k}+{t'_{N+1-k}}E^{\prime\theta}_{n,N+1-k}-(t'_k+t'_{N+1-k})E^{\prime\theta}_{n+1,k}\\
&=&A-(m-t'_k)E^{\prime\theta}_{n,k}+{(m-t'_k)}E^{\prime\theta}_{n,N+1-k}\prec A-mE^{\prime\theta}_{n,k}+mE^{\prime\theta}_{n+2,k}.
\end{eqnarray*}

Now consider a lower term in (\ref{eq-A12}) of the form
$$A''=(a''_{ij}):= A-\sum_{u=1}^Nt_{u}(E^{\prime\theta}_{n,u}-E^{\prime\theta}_{n+1,u})= A-\sum_{u=k}^{N+1-k}t_{u}(E^{\prime\theta}_{n,u}-E^{\prime\theta}_{n+1,u}).$$ 
The product $[mE^{\prime\theta}_{n,n+1}+\wla]\cdot[A'']$
is a linear combination of basis elements of the form
\begin{eqnarray}\label{eq-lv}
&&[A-\sum_{u=1}^Nt_{ u}(E^{\prime\theta}_{n, u}-E^{\prime\theta}_{n+1, u})+\sum_{v=1}^N l_{v}(E^{\prime\theta}_{n,v}-E^{\prime\theta}_{n+1,v})],
\end{eqnarray}
where all $l_v\in\N$ satisfy that $\sum l_v=m$ and $l_v+l_{N+1-v}\leq a''_{n+1,v}=t_v+t_{N+1-v}$ {\it for all} $v$. These, together with the fact that $\sum_v t_v=m$, force $l_v+l_{N+1-v}=t_{v}+t_{N+1-v}$ (forcing $l_v=0$ for all $v<k$ or $v>N+1-k$). Thus, $t_{n+1}=0$ implies $l_{n+1}=0$, and $\sum_{u=1}^Nt_{ u}E^{\prime\theta}_{n+1, u}=\sum_{u=1}^n(t_u+t_{N+1-u})E^{\prime\theta}_{n+1,u}=\sum_{v=1}^Nl_vE^{\prime\theta}_{n+1, v}$. Hence, the matrix in
(\ref{eq-lv}) belongs to $\Xinri$ or \eqref{eq-lv} becomes
$$ [A-\sum_{u=k}^{N+1-k}t_{u}E^{\prime\theta}_{n,u}+\sum_{v=k}^{N+1-k} l_{v}E^{\prime\theta}_{n,v}]\in \SnriZ .$$
which is clearly a term lower than $[A-mE^{\prime\theta}_{n,k}+mE^{\prime\theta}_{n+2,k}]$ (as $\bft\neq \bft^{(0)}$ or $t_k<m$). 

This completes the proof of the lemma.\end{proof}

We now return to the setup for $\SenrZ$.
Recall the dominance order on $\La(N,m)$: for $\mu,\nu\in\La(N,m)$,
\begin{equation}\label{dom order}
\mu\unlhd\nu\iff \mu_1\leq\nu_1,\mu_1+\mu_2\leq\nu_1+\nu_2,\ldots, \mu_1+\cdots+\mu_{N}\leq\nu_1+\cdots+\nu_{N},
\end{equation}
Observe that, if $\mu\unlhd\nu$, then
$
\mu_N=m-(\mu_1+\cdots+\mu_{N-1})\geq m-(\nu_1+\cdots+\nu_{N-1})=\nu_N$, $\mu_N+\mu_{N-1}\geq 
\nu_N+\nu_{N-1},\ldots, \mu_N+\cdots+\mu_1\geq \nu_N+\cdots+\nu_1$, Thus, if we set $\la^\tau:=(\la_N,\ldots,\la_1)$ for every $\la=(\la_1,\ldots,\la_N)\in\La(N,m)$, then we have
$$\mu\unlhd\nu\iff \mu^\tau\unrhd\nu^\tau.$$

For $h\in[1,n]$, recall \eqref{order yi} and let
$$\La(2n,m)_{h}:=\{\nu\in\La(2n,m)\mid \nu\leq \row_{h}(A)\}\neq\emptyset,$$
Parts (1) and (2) of the following result is a generalisation of the leading term formulas in \cite[Lem.~3.9]{BKLW}.

\begin{prop}\label{LThn}
Let $A=(a_{i,j})\in\Xienr$, $h\in[1,n]$, $\la\in\La(n,r)$ and $m$ a positive integer. Then the following formulas with leading terms hold in $\SenrZ$:
\begin{itemize}
\item[(1)] 
If $h\neq n$ and $A$ satisfies $a_{h,j}=0$, $a_{h+1,k}>0$, and $a_{h+1,j'}=0$ for all $j\geq k$, $j'>k$, 
then, for $\hla=\ro(A)-m\bfe^\theta_{h+1}$,
\begin{equation}\label{LT1a}
[mE_{h,h+1}^\theta+\hla]\cdot[A]=[A+\sum_{j=1}^k\nu_j^{(0)}(E_{h,j}^\theta-E_{h+1,j}^\theta)]+\text{\rm(lower terms)}_\preccurlyeq,
\end{equation}
where $\nu^{(0)}$ is the least element of $(\La(2n,m)_{h+1},\unlhd)$. In particular, if $a_{h+1,k}\geq m$, then
$\nu^{(0)}=(0,\ldots,0,\underset{(k)}m,0,\ldots,0)$ and \eqref{LT1a} becomes
\begin{equation}\label{LT1b}
[mE_{h,h+1}^\theta+\hla]\cdot[A]=[A+m(E_{h,k}^\theta-E_{h+1,k}^\theta)]+\text{\rm(lower terms)}_\preccurlyeq.
\end{equation}
\item[(2)] If $h\neq n$ and $A$ satisfies $a_{h,k}>0$, $a_{h,j}=0=a_{h+1,j'}$ for all $j< k$, $j'\leq k$, 
then, for $\hla=\ro(A)-m\bfe^\theta_{h}$,
\begin{equation}\label{LT2a}
[mE_{h+1,h}^\theta+\hla]\cdot[A]=[A+\sum_{j=k}^{2n}\nu^{(0')}_j(E_{h+1,j}^\theta-E_{h,j}^\theta)]+\text{\rm(lower terms)}_\preccurlyeq.
\end{equation}
where $\nu^{(0')}$ is the largest element of $(\La(2n,m)_{h},\unlhd)$. In particular, if $a_{h,k}\geq m$, then
$\nu^{(0')}=(0,\ldots,0,\underset{(k)}m,0,\ldots,0)$ and \eqref{LT2a} becomes
\begin{equation}\label{LT2b}
[mE_{h+1,h}^\theta+\hla]\cdot[A]=[A+m(E_{h+1,k}^\theta-E_{h,k}^\theta)]+\text{\rm(lower terms)}_\preccurlyeq.
\end{equation}
\item[(3)] If $A$ satisfies $a_{n,k}\geq m>0$ for some $k\in[1,n]$ and $a_{n,j}=0$, $\forall\ j<k$ and $j\geq 2n+1-k$,
then, for any $\hla=\ro(A)-m\bfe_n^\theta$, we have in $\SenrZ$
$$[mE^{\theta}_{n+1,n}+\hla]\cdot[A]=[A+m(E^{\theta}_{n+1,k}-E^{\theta}_{n,k})]+({\rm lower\ terms})_{\preccurlyeq}.$$
\end{itemize}
\end{prop}
\begin{proof} Part (3) is the $\fkf$-pullback of \eqref{trump1}.
We now prove (1). The proof of (2) is symmetric.

By the hypothesis,  each $\nu\in\La(2n,m)_{h+1}$ has the form $\nu=(\nu_1,\ldots,\nu_k,0,\ldots,0)$. Thus, by Proposition \ref{keyMFm}(1),
the left hand side of \eqref{LT1a} is a linear combination of $[A_\nu]$, where 
$A_\nu=A+\sum_{j=1}^k\nu_j(E_{h,j}^\theta-E_{h+1,j}^\theta)$. Putting $A_\nu=(m_{i,j})$ and $A_{\nu^{(0)}}=(m_{i,j}^{(0)})$, we have for all  $i<j$,
$$m_{i,j}=\begin{cases}a_{i,j}+\nu_j(\delta_{h,i}-\delta_{h+1,i}),&\text{ if }i\in[1,n];\\
a_{i,j}+\nu_{2n+1-j}(\delta_{2n+1-h,i}-\delta_{2h-h,i}),&\text{ if }i\in[n+1,2n].\\
\end{cases}$$
We now check \eqref{precN}.  For all $u<v$ with $u\in[1,n]$, since $\sum_{i\leq u}(\delta_{h,i}-\delta_{h+1,i})=\begin{cases} 1,&\text{ if }u=h;\\0,&\text{ otherwise.}\end{cases}$, it follows that
$$
\sum_{i\leq u<v\leq j}(m_{i,j}^{(0)}-m_{i,j})=\sum_{j\geq v}(\nu^{(0)}_j-\nu_j)\sum_{i\leq u}(\delta_{h,i}-\delta_{h+1,i})
=\begin{cases}\sum_{j=v}^k(\nu^{(0)}_{j}-\nu_{j}),&\text{if }u=h< v\leq k;\\
0,&\text{otherwise,}\end{cases}$$
which is non-negative since $(\nu^{(0)})^\tau\unrhd\nu^\tau$.

For all $u<v$ with $u\in[n+1,2n]$, since $\sum_{n<i\leq u}(\delta_{2n+1-h,i}-\delta_{2n-h,i})=\begin{cases} -1,&\text{ if }u=2n-h;\\
0,&\text{ otherwise }\end{cases}$, we have
$$\aligned
\sum_{i\leq u<v\leq j}(m_{i,j}^{(0)}-m_{i,j})&=\sum_{n< i\leq u\atop j\geq v}(m_{i,j}^{(0)}-m_{i,j})
=\sum_{j\geq v}(\nu^{(0)}_{2n+1-j}-\nu_{2n+1-j})\\
&\cdot\sum_{n<i\leq u}(\delta_{2n+1-h,i}-\delta_{2n-h,i})=\begin{cases}-\sum_{j'\leq 2n+1-v}(\nu^{(0)}_{j'}-\nu_{j'}),&\text{ if }u=2n-h\\
0,&\text{ otherwise,}\end{cases}
\endaligned$$
which is non-negative as $\nu^{(0)}\unlhd\nu$. This completes the proof of (1).
\end{proof} 
Note that the two special cases in \eqref{LT1b} and \eqref{LT2b} are extracted from \cite[Lem.~3.9]{BKLW} via the isomorphism $\fkf$ in \eqref{fkf}.\footnote{We modified their ``$\cdots=R$'' condition to a ``$\cdots\geq m$'' condition.}
Note also that the least/largest elements $\nu^{(0)}$ and $\nu^{(0')}$ in $\La(2n,m)_{h+1}$ and $\La(2n,m)_{h}$, respectively, have the form
$$\aligned
\nu^{(0)}&=(0,\ldots,0,a_{h+1,j}-a,a_{h+1,j+1},\ldots,a_{h+1,k},0,\ldots,0)\\
\nu^{(0')}&=(0,\ldots,0,a_{h,k},a_{h,k+1},\ldots,a_{h,j'+1},a_{h,j'}-a',0,\ldots,0),
\endaligned$$
where $j$ (resp. $j'$) is the index such that $\sum_{i=j+1}^ka_{h+1,i}< m\leq\sum_{i=j}^ka_{h+1,i}$ (resp., $\sum_{i=k}^{j'+1}a_{h,i}< m\leq\sum_{i=k}^{j'}a_{h,i}$) and 
$a=\sum_{i=j}^{k}a_{h+1,i}-m$ (resp., $a'=\sum_{i=k}^{j'}a_{h,i}-m$).

Let 
 \begin{equation}\label{set F}
 \mathscr T_{N}=\{(i,h,j)\mid 1\leq j\leq h<i\leq N\}.
 \end{equation}
  We order the set as in \cite[Thm.~3.10]{BKLW}:
\begin{equation}\label{order leq}
(i,h,j)\leq(i',h',j')\iff i<i'\text{ or }i=i',j<j'\text{ or }i=i',j=j', h>h'.
\end{equation}
For example, $\mathscr T_{4}$ has the following order:
$$\mathscr T_{4}=\{(2,1,1),(3,2,1),(3,1,1),(3,2,2),(4,3,1),(4,2,1),(4,1,1),(4,3,2),(4,2,2),(4,3,3)\}.$$

Like the construction in \cite[Thm.3.10]{BKLW}, for $A\in \Xienr$ and the largest element $(N,N-1,N-1)$ in $\mathscr T_{2n}$, where $N=2n$, let $D_{N,N-1,N-1}=\diag(\co(A)-a_{N,N-1}\bfe^\theta_{N-1})$ so that $\co(D_{N,N-1,N-1}+a_{N,N-1}E^\theta_{N,N-1})=\co(A)$. Inductively, suppose $D_{i',h',j'}$ is defined and $(i,h,j)$ is an immediate predecessor of $(i',h',j')$, then the diagonal matrix $D_{i,h,j}$ is uniquely defined by the equation
$$\co(D_{i,h,j}+a_{i,j}E_{h+1,h}^\theta)=\ro(D_{i',h',j'}+a_{i',j'}E_{h'+1,h'}^\theta).$$
Note that, for the least element $(2,1,1)$ in $\mathscr T_{2n}$, $D_{2,1,1}$ is the diagonal matrix uniquely defined by  
$\co(D_{2,1,1}+a_{2,1}E_{2,1}^\theta)=\ro(D_{3,2,1}+a_{3,1}E_{3,2}^\theta).$ In particular, $\ro(D_{2,1,1}+a_{2,1}E_{2,1}^\theta)=\ro(A)$.

The following result is the type C counterpart of  similar results for type B in \cite[Thm. 3.10]{BKLW} and for type $A$ in \cite[3.5]{BLM}. In \cite[\S5.4]{BKLW}, a similar triangular relation is developed via the twin products $[D_{i,n+1,j}+a_{i,j}E^{\prime\theta}_{n+2,n+1}][D_{i,n,j}+a_{i,j}E^{\prime\theta}_{n+1,n}]$ within the centraliser subalgebra $\SnriZ$ of $\SnrZ$. Our version is independent of $\SnriZ$ and simplify 
their version by using only the leading term of a twin product, i.e.,  by dropping a long tail. This becomes possible thanks to the new leading term formulas in Proposition \ref{LThn}(3).

\begin{theo}\label{TR}
For any $A=(a_{i,j})\in\Xienr$, the following triangular relation holds in $\SenrZ$:
\begin{eqnarray}\label{TR0}
\sfm(A):=\prod_{(i,h,j)\in(\mathscr T_{2n},\leq)}[D_{i,h,j}+a_{i,j}E^{\theta}_{h+1,h}]=[A]+(\text{lower terms})_\preccurlyeq.
\end{eqnarray}
\end{theo}
\begin{proof} Similar to that of the type $A/B$ case, the proof is standard by repeatedly using the leading term formulas in Proposition \ref{LThn}.
We use the example for $n=2$ above to illustrate it. From the set $\mathscr T_4$ above, there are 10 factors in the product: 
 \begin{equation*}\label{TRex}
 \aligned{}
 [D_{2,1,1}+a_{2,1}E^{\theta}_{2,1}]
 &\cdot_{9}[D_{3,2,1}+a_{3,1}E^{\theta}_{3,2}]\cdot_8[D_{3,1,1}+a_{3,1}E^{\theta}_{2,1}]\cdot_7[D_{3,2,2}+a_{3,2}E^{\theta}_{3,2}]\\
 &\cdot_6[D_{4,3,1}+a_{4,1}E^{\theta}_{1,2}]\cdot_5[D_{4,2,1}+a_{4,1}E^{\theta}_{3,2}]\cdot_4[D_{4,1,1}+a_{4,1}E^{\theta}_{2,1}]\\
& \cdot_3[D_{4,3,2}+a_{4,2}E^{\theta}_{1,2}]\cdot_2[D_{4,2,2}+a_{4,2}E^{\theta}_{3,2}]\cdot_1[D_{4,3,3}+a_{4,3}E^{\theta}_{1,2}]\cdot_0[\co(A)]
\endaligned
 \end{equation*}
 Here we have used the fact $E^\theta_{h+1,h}=E^\theta_{4-h,5-h}$. We make the product in the order as indicated. Each step gives a leading term by the lemma above. The following 10 matrices is the leading term of the 10 multiplications ordered from 0 to 9. Recall $a_{i,j}=a_{5-1,5-j}$.
$$
\left(\begin{smallmatrix}\bullet&\ &a_{12}&&\cdot&&\cdot\\
&&{\;\uparrow_0}&&&&\\
\cdot&&\bullet&&\cdot&&\cdot\\
&&&&&&\\
\cdot&&\cdot&&\bullet&&\cdot\\
&&&&\;\downarrow^0&&\\
\cdot&&\cdot&&a_{43}&&\bullet\\
\end{smallmatrix}\right),
\left(\begin{smallmatrix}\bullet&\ &a_{12}&&\cdot&&\cdot\\
&&&&&&\\
\cdot&&\bullet&&a_{13}&&\cdot\\
&&\;\downarrow^1&&{\;\uparrow_1}&&\\
\cdot&&a_{42}&&\bullet&&\cdot\\
&&&&&&\\
\cdot&&\cdot&&a_{43}&&\bullet\\
\end{smallmatrix}\right),
\left(\begin{smallmatrix}
\bullet&\ &a_{12}&&a_{13}&&\cdot\\
&&&&{\;\uparrow_2}&&\\
\cdot&&\bullet&&\cdot&&\cdot\\
&&&&&&\\
\cdot&&\cdot&&\bullet&&\cdot\\
&&\;\downarrow^2&&&&\\
\cdot&&a_{42}&&a_{43}&&\bullet\\
\end{smallmatrix}\right),
\left(\begin{smallmatrix}
\bullet&\ &a_{12}&&a_{13}&&\cdot\\
\;\downarrow^3&&&&&&\\
a_{41}&&\bullet&&\cdot&&\cdot\\
&&&&&&\\
\cdot&&\cdot&&\bullet&&a_{14}\\
&&&&&&{\;\uparrow_3}\\
\cdot&&a_{42}&&a_{43}&&\bullet\\
\end{smallmatrix}\right),
\left(\begin{smallmatrix}
\bullet&\ &a_{12}&&a_{13}&&\cdot\\
&&&&&&\\
\cdot&&\bullet&&\cdot&&a_{14}\\
\;\downarrow^4&&&&&&{\;\uparrow_4}\\
a_{41}&&\cdot&&\bullet&&\cdot\\
&&&&&&\\
\cdot&&a_{42}&&a_{43}&&\bullet\\
\end{smallmatrix}\right),
$$
$$
\left(\begin{smallmatrix}
\bullet&\ &a_{12}&&a_{13}&&a_{14}\\
&&&&&&{\;\uparrow_5}\\
\cdot&&\bullet&&\cdot&&\cdot\\
&&&&&&\\
\cdot&&\cdot&&\bullet&&\cdot\\
\;\downarrow^5&&&&&&\\
a_{41}&&a_{42}&&a_{43}&&\bullet\\
\end{smallmatrix}\right),
\left(\begin{smallmatrix}
\bullet&\ &a_{12}&&a_{13}&&a_{14}\\
&&&&&&\\
\cdot&&\bullet&&a_{23}&&\cdot\\
&&\;\downarrow^6&&\;\uparrow_6&&\\
\cdot&&a_{32}&&\bullet&&\cdot\\
&&&&&&\\
a_{41}&&a_{42}&&a_{43}&&\bullet\\
\end{smallmatrix}\right),
\left(\begin{smallmatrix}
\bullet&\ &a_{12}&&a_{13}&&a_{14}\\
\;\downarrow^7&&&&&&\\
a_{31}&&\bullet&&a_{23}&&\cdot\\
&&&&&&\\
\cdot&&a_{32}&&\bullet&&a_{24}\\
&&&&&&\;\uparrow_7\\
a_{41}&&a_{42}&&a_{43}&&\bullet\\
\end{smallmatrix}\right),
\left(\begin{smallmatrix}
\bullet&\ &a_{12}&&a_{13}&&a_{14}\\
&&&&&&\\
\cdot&&\bullet&&a_{23}&&a_{24}\\
\;\downarrow^8&&&&&&\;\uparrow_8\\
a_{31}&&a_{32}&&\bullet&&\cdot\\
&&&&&&\\
a_{41}&&a_{42}&&a_{43}&&\bullet\\
\end{smallmatrix}\right),
$$
$$
\left(\begin{smallmatrix}
\bullet&\ &a_{12}&&a_{13}&&a_{14}\\
\;\downarrow^9&&&&&&\\
a_{21}&&\bullet&&a_{23}&&a_{24}\\
&&&&&&\\
a_{31}&&a_{32}&&\bullet&&a_{34}\\
&&&&&&\;\uparrow_9\\
a_{41}&&a_{42}&&a_{43}&&\bullet\\
\end{smallmatrix}\right)=A,
$$
Here the $j$th $\bullet$ on the diagonal of the first matrix is the sum of $j$th column of $A$ which decreases to $\bullet=a_{j,j}$ in the last matrix, an arrow $\uparrow_i$ or $\downarrow^i$ below or above an entry tells how the entry is moved up or down, and $i$ indicates the $i$th multiplication. Note that steps 1,4,6, and 8 used the leading term formula in Proposition \ref{LThn}(3).
\end{proof}

\begin{coro}The $\sZ$-algebra $\SenrZ$ is generated by the elements $[\hla]$ $(\la\in\La(n,r))$, $[E^\theta_{h,h+1}+\hmu]$ $(h\in[1,2n), \mu\in\Lambda(n,r-1))$, and has a new basis $\{\sfm(A)\mid A\in\Xi_{2n,2r}\}$ which is triangularly related to the basis $\{[A]\mid A\in\Xi_{2n,2r}\}$. 
\end{coro}

The following result requires the generalised leading term in Proposition \ref{LThn}(1)\&(2) and will be needed in \S7.

\begin{coro}\label{TR1}
Maintain the notations in Theorem \ref{TR}. If several factors of the form $[D_{i,n,j}+a_{i,j}E^\theta_{n+1,n}]$ in the product (\ref{TR0}) are replaced by $[D_{i,n,j}'+(a_{i,j}-s)E^\theta_{n+1,n}]$ with $0<s\leq a_{i,j}$, then the resulting product is either 0 or a linear combination of $[B]$ with $B\prec A$.
\end{coro}
\begin{proof} The multiplication by $[D_{i,n,j}'+(a_{i,j}-s)E^\theta_{n+1,n}]$, if nonzero, moves $a_{i,j}-s$ from $(n,j)$ position to $(n+1,j)$ position. Thus, the resulting leading term is (strictly) $\prec$ the corresponding leading term when multiplied by $[D_{i,n,j}+a_{i,j}E^\theta_{n+1,n}]$. Now applying \eqref{LT1a} and \eqref{LT2a} to the multiplication by the next factor(s) produces leading terms that are $\prec$ the corresponding leading term in the original product. 
\end{proof}
 
\section{Long multiplication formulas in $\SenrQ$}
Recall the elements in \eqref{FirstAjr} that define the actions of the generators of $\bfUin$ on the tensor space
$\Omega_{2n}^{\otimes r}$. We now introduce the general version of such elements and extend the short multiplication formulas given in \S4 to these ``long'' elements defined below.  

Recall from \eqref{hla} the map $\la\mapsto\hla$ from $\La(n,r)$ to $\La(2n,2r)$ and, for $\bfj,\bfj'\in\Z^N$, let
$$\bfj\centerdot\bfj'=j_1j_1'+j_2j_2'+\cdots+j_Nj_N'.$$
Associated with $A\in\Xienl$ and $\bfj\in\Z^{2n}$, we define the following elements in $\SenrZ$  (cf. \cite[(4.1.1)]{DW}): 
\begin{equation}\label{Ajr}
A({\bf j},r)=\begin{cases}\displaystyle
\sum_{\la\in\Lambda(n,r-\frac{|A|}2)}\up^{\hla\centerdot\bfj}[A+{\hla}],&\text{ if }|A|\leq 2r,\\
0,\qquad&\text{ if }|A|>2r,\end{cases}
\end{equation}

 For any $\bfj=(j_1,\ldots,j_n,j_{n+1},\ldots,j_{2n})\in\Z^{2n}$, let
 \begin{equation}\label{j*}
 \bfj^*=(j_1+j_{2n},\ldots,j_n+j_{n+1})\in\Z^n,\quad\bfj^\circ=(j_1,\ldots,j_n,0,j_{n+1},\ldots,j_{2n})\in\Z^{2n+1},
 \end{equation}
 Then $\hla\centerdot\bfj=\la\centerdot\bfj^*$. Thus, $A(\bfj,r)=A(\bfj^*,r)$ where $\bfj^*$ is regarded as an element of $\Z^{2n}$ by adding $n$ zeros at the end. In particular, if $O$ denotes the zero matrix in $\Xi_{2n}$, then $O(\bfe_i,r)=O(\bfe_{2n+1-i},r)$ for all $i\in[1,n]$.
 
 Recall the idempotent $e=\sum_{\la\in\La(n,r)}[\diag(\hla^\dagger)]\in\SnrZ$ and the algebra isomorphism
 $\fkf:\SenrZ\to e\SnrZ e=\SnriZ$ in \eqref{fkf}.
\begin{lemm}\label{fkf2}
For $A\in\Xienrl$, $\bfj\in\Z^{2n}$, let $A^\dagger$ be defined as in \eqref{Adag} and $A^\dagger(\bfj^\circ,r)\in\SnrZ$ defined in \cite[(4.1.1)]{DW}. Then $A(\bfj,r)=\fkf^{-1}(eA^\dagger(\bfj^\circ,r)e)$.
\end{lemm}
\begin{proof}By definition, we may assume $|A|\leq 2r$ and so $A^\dagger(\bfj^\circ,r)=
\sum_{\mu\in\Lambda(n+1,r-\frac{|A|}2)}\up^{\wmu\centerdot\bfj^\circ}[A^\dagger+{\wmu}]$, where $\wmu=(\mu_1,\dots,\mu_n,2\mu_{n+1}+1,\mu_n,\ldots,\mu_1)$, then 
$$e[A^\dagger+{\wmu}]e=\begin{cases}0,&\text{ if }\mu_{n+1}>0;\\
[A^\dagger+\hla^\dagger],&\text{ if }\mu_{n+1}=0,\end{cases}$$
 where $\la=(\mu_1,\ldots,\mu_n)$. The assertion now follows from the fact that $\wmu\centerdot\bfj^\circ=\hla\centerdot\bfj$.
\end{proof}

For $1\leq h\leq n$, put 
\begin{equation}\label{al-}
\al_h=\bfe_{h}-\bfe_{h+1},\;\;\al^-_h=-\bfe_h-\bfe_{h+1}\in\mathbb Z^{2n}.
\end{equation}
\begin{theo}\label{longMF}Maintain the notations introduced above. 
 For $A=(a_{i,j})\in\Xienl$, $h\in[1,n)$, and ${\bf j}=(j_1,j_2,\dots,j_{2n})\in \mathbb{Z}^{2n}$, let $\b_p(A,h),\b'_p(A,h)$ be defined as in \eqref{beta}. Then
  the following multiplication formulas hold in $\SenrQ$ for all $r\geq \frac{|A|}2$:
\begin{itemize}
\item[] \vspace{-3ex}
$$(1)\;\;{ O}({\bf j},r)A({\bf j'},r)=\up^{\ro(A)\centerdot{\bf j}}A({\bf {j+j'}},r),\quad A({\bf j'},r){O}({\bf j},r)=\up^{\co(A)\centerdot\bfj}A({\bf {j+j'}},r);\qquad\qquad\quad\;\;\quad$$
\item[]\vspace{-3ex}
\begin{equation*}
\aligned
(2)\;\;E^{\theta}_{h,h+1}({\bf 0},r)&\cdot A({\bf j},r)=\sum_{\substack{1\leq p<h\\ a_{h+1,p}\geq 1}}\up^{\b_p (A,h)}\overline{[\![a_{h,p}+1]\!]}
(A+E^{\theta}_{h,p}-E^{\theta}_{h+1,p})({\bf j}+\alpha_{h},r)\\
&+\varepsilon \frac{\up^{\b_h(A,h)-j_h-j_{2n+1-h}}}{\up-\up^{-1}}\Big((A-E^{\theta}_{h+1,h})({\bf j}+\al _h,r)-(A-E^{\theta}_{h+1,h})({\bf j}+\al^-_h,r)\Big)\\
&+ \up^{\b_{h+1}(A,h)+j_{h+1}+j_{2n-h}}\overline{[\![a_{h,h+1}+1]\!]}(A+E^{\theta}_{h,h+1})({\bf j},r)\\
&+\sum_{\substack{h+1<p\leq 2n\\ a_{h+1,p}\geq 1}}\up^{\b_p(A,h)}\overline{[\![a_{h,p}+1]\!]}
(A+E^{\theta}_{h,p}-E^{\theta}_{h+1,p})({\bf j},r)\\
\endaligned
\end{equation*}
where $\varepsilon=\delta^\leq_{1,a_{h+1,h}}$ is given in  \eqref{c_A}.
\item[]\vspace{-3ex}
\begin{equation*}
\aligned
(3)\, E^{\theta}_{h+1,h}(&{\bf 0},r)\cdot A({\bf j},r)=
\sum_{\substack{1\leq p<h\\ a_{h,p}\geq 1}}\up^{\b'_p(A,h)}\overline{[\![a_{h+1,p}+1]\!]}
(A-E^{\theta}_{h,p}+E^{\theta}_{h+1,p})({\bf j},r)\\
&+ \up^{\b'_h(A,h)+j_{h}+j_{2n+1-h}}\overline{[\![a_{h+1,h}+1]\!]}(A+E^{\theta}_{h+1,h})({\bf j},r)\\
&+\varepsilon'\frac{\up^{\b'_{h+1}(A,h)-j_{h+1}-j_{2n-h}}}{\up-\up^{-1}}\big((A-E^{\theta}_{h,h+1})({\bf j}-\al _h,r)-(A-E^{\theta}_{h,h+1})({\bf j}+\al^-_h,r)\big)\\
&+\sum_{\substack{h+1<p\leq 2n\\ a_{h,p}\geq 1}}\up^{\b'_p(A,h)}\overline{[\![a_{h+1,p}+1]\!]}
(A-E^{\theta}_{h,p}+E^{\theta}_{h+1,p})({\bf j}-\al_h,r),
\endaligned
\end{equation*}
where $\varepsilon'=\delta_{1,a_{h,h+1}}^\leq$.
\item[]\vspace{-3ex}
 \begin{equation*}\aligned
(4)\;E^{\theta}_{n+1,n}({\bf 0},&r)\cdot A({\bf j},r)=\sum_{i\neq n,n+1\atop a_{n,i}\geq1}\up^{\b'_i(A,n)-\delta^\leq_{n+1,i}}\overline{[\![a_{n+1,i}+1]\!]}(A-E^{\theta}_{n,i}+E^{\theta}_{n+1,i})({\bf j},r)\\
&\quad+\up^{\b'_n(A,n)+j_n+j_{n+1}}\overline{[\![a_{n+1,n}+1]\!]}(A+E^{\theta}_{n+1,n})({\bf j},r)\\
&\quad+\delta^\leq_{1,a_{n,n+1}}\frac{\up^{\b'_{n+1}(A,n)-j_{n+1}-j_n}}{\up-\up^{-1}}\Big((A-E^{\theta}_{n,n+1})({\bf j},r)-(A-E^{\theta}_{n,n+1})({\bf j}-\bfe^\theta_n,r)\Big)\\
&\quad+ c_AA({\bf j}-\bfe_n,r)\\
\endaligned
\end{equation*}
 where $c_A$ is defined in \eqref{c_A}.
\end{itemize}
\end{theo}
\begin{proof} For the $h<n$ case, by using Lemma \ref{keyMF}(1)\&(2), the formulas (1)--(3) can be proved similarly to that of \cite[Thm. 4.2]{DW}. Alternatively, they can also be obtained by applying Lemma \ref{fkf2} to the formulas (1)--(3) in \cite[Thm. 4.2]{DW} using $O^\dagger({\bf j}^\circ,r)$,
$E^{\theta\dagger}_{h,h+1}({\bf 0}^\circ,r), E^{\theta\dagger}_{h+1,h}({\bf 0}^\circ,r)$, and $A^\dagger({\bf j}^\circ,r)$.

We now use Lemma \ref{keyMF}(3) to prove (4), the $h=n$ case. 
 By the definition of $A(\bfj,r)$, 
\begin{equation}\label{expand}
\aligned
E^{\theta}_{n,n+1}({\bf 0},r)\cdot A({\bf j},r)&=\Big(\sum_{\nu\in \La(n,r-1)}[E^{\theta}_{n,n+1}+\hnu]\Big)\Big(\sum_{\mu \in \La(n,r-\frac{|A|}2)}\up^{\hmu\centerdot{\bf j}}[A+\hmu]\Big)\\
&=\sum_{\mu \in \La(n,r-\frac{|A|}2)}\up^{\hmu\centerdot{\bf j}}\bigg(\sum_{\nu\in \La(n,r-1)}[E^{\theta}_{n,n+1}+\hnu]\bigg)[A+\hmu]\\
&=\sum_{\mu\in\La(n,r-\frac{|A|}2)}\up^{\hmu.{\bf j}} [E^{\theta}_{n,n+1}-E^{\theta}_{n,n}+\ro(A)+\hmu]\cdot[A+\hmu].
\endaligned
\end{equation}
Here, by \eqref{weight idemp}, the only nonzero terms satisfy $\co(E^{\theta}_{n,n+1})+\hnu=\ro(A)+\hmu$.
Write $A+\hmu=(a^{\hmu}_{i,j})\in\Xienr$. By Lemma \ref{keyMF}(3),
\begin{eqnarray*}
&& [E^{\theta}_{n,n+1}-E^{\theta}_{n,n}+\ro(A)+\hmu]\cdot[A+\hmu]\\
&=&c_{A+\hmu}[A+\hmu]+\sum_{i\in[1,2n],a_{n,i}^\hmu\geq1}\up^{\b'_i(A+\hmu,n)-\delta_{n+1,i}^\leq}\overline{[\![a^{\hmu}_{n+1,i}+1]\!]}[A+\hmu-E^{\theta}_{n,i}+E^{\theta}_{n+1,i}].
\end{eqnarray*}
Since $a^{\hmu}_{i,i}=\hmu_{i}$, $a^{\hmu}_{i,j}=a_{i,j}$ for $i\neq j$, and $\hmu_{n+1}=\hmu_n$, it follows that $\b'_p(A+\hmu,n)=\b'_p(A,n)$ for all $p\in[1,2n]$, and $c_{A+\hmu}=\up^{-\hmu_n}c_A$.
Hence, substituting into \eqref{expand} gives the last term in (4):
$$\sum_{\mu\in\La(n,r-\frac{|A|}2)}\up^{\hmu.{\bf j}}\up^{-\hmu_n}c_A[A+\hmu]=
c_A\sum_{\mu\in\La(n,r-\frac{|A|}2)}\up^{\hmu\centerdot{(\bfj-\bfe_n)}}[A+\hmu]=c_AA(\bfj-\bfe_n,r).$$
On the other hand, substituting the summation into \eqref{expand} yields
\begin{equation}\label{zzz}
\aligned
&\quad\sum_{\mu\in\La(n,r-\frac{|A|}2)}\up^{\hmu\centerdot{\bf j}}\sum_{i\in[1,2n]\atop a^{\hmu}_{n,i}\geq1}\up^{\b'_i(A+\hmu,n)-\delta_{n+1,i}^\leq}\overline{[\![a^{\hmu}_{n+1,i}+1]\!]}[A+\hmu-E^{\theta}_{n,i}+E^{\theta}_{n+1,i}]\\
&=\sum_{i\neq n,n+1\atop a_{n,i}\geq1}\up^{\b'_i(A,n)-\delta_{n+1,i}^\leq}\overline{[\![a_{n+1,i}+1]\!]}\sum_{\mu\in\La(n,r-\frac{|A|}2)}\up^{\hmu.{\bf j}}[A+\hmu-E^{\theta}_{n,i}+E^{\theta}_{n+1,i}]\\
&+\sum_{\mu\in\La(n,r-\frac{|A|}2),\hmu_n\geq1}\up^{\b'_n(A,n)+\hmu.{\bf j}}\overline{[\![a_{n+1,n}+1]\!]}[A+\hmu-E^{\theta}_{n,n}+E^{\theta}_{n+1,n}]\\
&+\delta^\leq_{1,a_{n,n+1}}\sum_{\mu\in\La(n,r-\frac{|A|}2)}\up^{\b'_{n+1}(A,n)-1+\hmu.{\bf j}}\overline{[\![\hmu_{n+1}+1]\!]}[A+\hmu-E^{\theta}_{n,n+1}+E^{\theta}_{n+1,n+1}].
\endaligned
\end{equation}
Since $|A|=|A-E^{\theta}_{n,i}+E^{\theta}_{n+1,i}|$, the inner summation of the double sum term in \eqref{zzz} gives, for $i\neq n,n+1$, $(A-E^{\theta}_{n,i}+E^{\theta}_{n+1,i})(\bfj,r)$. So, the double sum gives the summation term in (4).
For the second summation term of \eqref{zzz}, since $E^\theta_{n,n}=\diag(\bfe_n^\theta)$, we have
\begin{eqnarray*}
&&\sum_{\mu\in\La(n,r-\frac{|A|}2),\hmu_n\geq1}\up^{\b'_n(A,n)+\hmu.{\bf j}}\overline{[\![a_{n+1,n}+1]\!]}[A+E^{\theta}_{n+1,n}+\hmu-\bfe_n^\theta]\\
&=&\up^{\b'_n(A,n)+j_n+j_{n+1}}\overline{[\![a_{n+1,n}+1]\!]}\sum_{\mu\in\La(n,r-\frac{|A|}2),\hmu_n\geq1}\up^{(\hmu-\bfe_n^\theta).{\bf j}}[A+E^{\theta}_{n+1,n}+\hmu-\bfe_n^\theta]\\
&=&\up^{\b'_n(A,n)+j_n+j_{n+1}}\overline{[\![a_{n+1,n}+1]\!]}(A+E^{\theta}_{n+1,n})({\bf j},r),
\end{eqnarray*}
giving the second term in (4).

Finally, if $a_{n,n+1}\geq1$, then, by noting $E^\theta_{n+1,n+1}=E^\theta_{n,n}=\diag(\bfe_n^\theta)$ and $\hmu_n=\hmu_{n+1}$, the last summation in \eqref{zzz} has the form
\begin{eqnarray*}
&&\sum_{\mu\in\La(n,r-\frac{|A|}2)}\up^{\b'_{n+1}(A,n)-1+\hmu.{\bf j}}\overline{[\![\hmu_{n+1}+1]\!]}[A+\hmu-E^{\theta}_{n,n+1}+\bfe_n^\theta]\\
&=&\frac{\up^{\b'_{n+1}(A,n)-1}}{1-\up^{-2}}\sum_{\mu\in\La(n,r-\frac{|A|}2)}\up^{\hmu.{\bf j}}(1-\up^{-2(\hmu_{n+1}+1)})[A-E^{\theta}_{n,n+1}+\hmu+\bfe_n^\theta]\\
&=&\frac{\up^{\b'_{n+1}(A,n)-j_{n+1}-j_n}}{\up-\up^{-1}}\sum_{\mu\in\La(n,r-\frac{|A|}2)}\up^{(\hmu+\bfe_n^\theta)\centerdot{\bf j}}(1-\up^{-2(\hmu_{n+1}+1)})[A-E^{\theta}_{n,n+1}+\hmu+\bfe_n^\theta]\\
&=&\frac{\up^{\b'_{n+1}(A,n)-j_{n+1}-j_n}}{\up-\up^{-1}}\sum_{\mu\in\La(n,r-\frac{|A|}2)}\big(\up^{(\hmu+\bfe_n^\theta).{\bf j}}-\up^{(\hmu+\bfe_n^\theta).({\bf j}-\bfe_n^\theta)}\big)[A-E^{\theta}_{n,n+1}+\hmu+\bfe_n^\theta]\\
&=&\frac{\up^{\b'_{n+1}(A,n)-j_{n+1}-j_n}}{\up-\up^{-1}}\sum_{\la\in\La(n,r-\frac{|A|}2+1)}\big(\up^{\hla.{\bf j}}-\up^{\hla.({\bf j}-\bfe_n^\theta)}\big)[A-E^{\theta}_{n,n+1}+\hla],\\
\end{eqnarray*}
giving the third term in (4). This completes the proof of the theorem.
\end{proof}

We may now compute certain divided powers  in $\SenrQ$.
\begin{coro}\label{DP} Let $m$ be a positive integer. 
\begin{itemize}
\item[(1)] If $h\in[1,n)$, then we have, for all $r\geq m$, 
$$\frac{E^\theta_{h,h+1}(\bfl,r)^m}{[m]^!}=(mE^\theta_{h,h+1})(\bfl,r),\quad \frac{E^\theta_{h+1,h}(\bfl,r)^m}{[m]^!}=(mE^\theta_{h+1,h})(\bfl,r).$$
\item[(2)]
If $h=n$, then
 there exist $f_{s,t}\in\Q(\up)$, $\bfj_{s,t}\in\Z^{2n}$ for $s\in[0,m),1\leq t\leq n_s$, independent of $r\geq m$, such that
$$\frac{E^{\theta}_{n,n+1}({\bf 0},r)^m}{[m]^!}=(mE^{\theta}_{n,n+1})({\bf 0},r)+\sum_{s=0}^{m-1}\sum_{t=1}^{n_s}f_{s,t}(sE^{\theta}_{n,n+1})(\bfj_{s,t},r).$$ 
\item[(3)] For $A=(a_{i,j})\in\Xienl$ and ${\bf j}=(j_1,j_2,\dots,j_{2n})\in \mathbb{Z}^{2n}$, there exist finitely many 
$B_a\in\Xienl$, $\bfj^{(b)}\in\Z^{2n}$, and
$g_{B_a,\bfj^{(b)}}\in\Q(\up)$ such that, for all $r\geq\frac{|A|}2$, 
$$(mE_{n+1,n}^{\theta})(\bfl,r)\cdot A(\bfj,r)=\sum_{a,b}g_{B_a,\bfj^{(b)}}B_a(\bfj^{(b)},r).$$
\end{itemize}
\end{coro}
\begin{proof}The first assertion is clear; see the proof of \cite[Cor.~4.5]{DW}. We now prove the $h=n$ case.
For $A=E^{\theta}_{n,n+1}$, we first observe that, for a positive integer $a$,
$$\b'_p(aE^{\theta}_{n,n+1},n)=\sum_{l\leq p}a_{n+1,l}-\sum_{l<p }a_{n,l} =\begin{cases}0,&\text{ if }p<n;\\
a, &\text{ if }p=n,n+1;\\
0,&\text{ if }p>n+1.\\
\end{cases}$$
Thus, by Theorem \ref{longMF}(4) and \eqref{cE}, we have, for any positive integer $a$,
\begin{equation*}\label{EE}
\aligned
E^{\theta}_{n,n+1}({\bf 0},r)&(aE^{\theta}_{n,n+1})({\bfj},r)\\
&=\up^{a+j_n+j_{n+1}}\overline{[\![a+1]\!]}((a+1)E^{\theta}_{n+1,n})({\bfj},r)\\
&\quad+\frac{\up^{a-j_{n+1}-j_n}}{\up-\up^{-1}}\big(((a-1)E^{\theta}_{n,n+1})(\bfj,r)-((a-1)E^{\theta}_{n,n+1})(\bfj-\bfe_n^\theta,r)\big)\\
&\quad+(\up^a-\up^{-a})(aE^{\theta}_{n,n+1})(\bfj-\bfe_n,r).
\endaligned\end{equation*}
Hence, for $\bfj=\bfl$, we have
\begin{equation}\label{EaE}
\aligned
E^{\theta}_{n,n+1}({\bf 0},r)&(aE^{\theta}_{n,n+1})({\bfl},r)\\
&=[a+1]((a+1)E^{\theta}_{n+1,n})({\bfl},r)+(\up^a-\up^{-a})(aE^{\theta}_{n,n+1})(-\bfe_n,r)\\
&\quad+\frac{\up^{a}}{\up-\up^{-1}}\big(((a-1)E^{\theta}_{n,n+1})(\bfl,r)-((a-1)E^{\theta}_{n,n+1})(-\bfe_n^\theta,r)\big).
\endaligned\end{equation}
Now assertion (2) follows easily from an induction on $m$.

Finally, for assertion (3), let $E^{\theta}_{n,n+1}({\bf 0},r)^{(m)}:=\frac{E^{\theta}_{n,n+1}({\bf 0},r)^m}{[m]^!}$. By (2), we write $(mE_{n+1,n}^{\theta})(\bfl,r)$ as a linear combination of $E^{\theta}_{n,n+1}({\bf 0},r)^{(m-i)}O(\bfj_{i,a},r)$ ($i\in[0,m],a\in[1,p_i]$) and so (3) follows from Theorem \ref{longMF}(1),(4).
\end{proof}

\begin{coro}\label{CF}
For $h\in[1,n-1)$ and $\bfj\in\Z^{2n}$, we have
$$\aligned
(1)\,\quad\qquad O(\bfj,r)E_{n,n+1}^\theta(\bfl,r)&=E_{n,n+1}^\theta(\bfl,r)O(\bfj,r),\\
(2)\quad E_{h,h+1}(\bfl,r)E_{n,n+1}^\theta(\bfl,r)&=E_{n,n+1}^\theta(\bfl,r)E_{h,h+1}(\bfl,r),\\
(3)\quad E_{h+1,h}(\bfl,r)E_{n,n+1}^\theta(\bfl,r)&=E_{n,n+1}^\theta(\bfl,r)E_{h+1,h}(\bfl,r).
\endaligned$$
\end{coro}
\begin{proof} By Theorem \ref{longMF}(1), $O(\bfj,r)E_{n,n+1}^\theta(\bfl,r)=\up^{\bfj\centerdot\bfe^\theta_n}E_{n,n+1}(\bfj,r)=E_{n,n+1}^\theta(\bfl,r)O(\bfj,r),$ proving (1).
For (2), since $h<n-1$,  all $\b_{h+1}(E^{\theta}_{n,n+1},h), \b'_n(E^{\theta}_{h,h+1},n)$, and $c_{E^{\theta}_{h,h+1}}$ are 0. Thus, by Theorem \ref{longMF}(2)\&(4),
\begin{eqnarray*}
E^{\theta}_{h, h+1}({\bf 0},r)E^{\theta}_{n,n+1}({\bf 0},r)
&=&\up^{\b_{h+1}(E^{\theta}_{n,n+1},h)}(E^{\theta}_{n,n+1}+E^{\theta}_{h,h+1})({\bf 0},r)=(E^{\theta}_{n,n+1}+E^{\theta}_{h,h+1})({\bf 0},r)\\
&=&\up^{\b'_n(E^{\theta}_{h+1,h},n)}(E^{\theta}_{h,h+1}+E^{\theta}_{n+1,n})({\bf 0},r)+c_{E^{\theta}_{h,h+1}} E^{\theta}_{h,h+1}(-\bfe_n,r)\\
&=&E^{\theta}_{n,n+1}({\bf 0},r)E^{\theta}_{h,h+1}({\bf 0},r),
\end{eqnarray*}
as desired. The proof of (3) is similar.
\end{proof}
We end this section with an application by showing the map in \eqref{rho3} is surjective. For $i\in[1,n]$, $m\in\N$, let
$
k_{i,r}:=
O(\bfe_i,r),$ and $\big[{k_{i,r};0\atop m}\big]$ as defined in \eqref{Kbinom}.
 Then, for any $\la\in\La(n,r)$, we have in $\SenrZ$
\begin{equation}\label{k over la}
\Big[{k\atop\la}\Big]:=\prod_{i=1}^{n}\begin{bmatrix}k_{i,r};0\\\l_{i,r}\end{bmatrix}=[\hla].
\end{equation}
See \cite[Lemma 6.4]{DW} for a proof.
\begin{coro} \label{rho4}
The $\rho^\imath_{r}$ in \eqref{rho3} is surjective.
\end{coro}
\begin{proof} By Corollaries \ref{DP} and \ref{motivation ex}, elements $[\hla]$, $(mE^{\theta}_{h+1,h})(\bfl,r)$, $(mE^{\theta}_{h,h+1})(\bfl,r)$ are all in the image of $\tilde\eta_r\circ\rho^\imath_{r}$, where $\tilde\eta_r$ is given in \eqref{teta}. Thus, for any $A\in\Xienr$, Theorem \ref{TR} implies that
$$\prod _{(i,h,j)\in(\mathscr T_{2n},\leq)}(a_{i,j}E^{\theta}_{h+1,h})(\bfl,r)\cdot[\co(A)]=[A]+(\text{lower terms})_\preccurlyeq,$$
where the product is taken over the order $\leq$ on the set $\mathscr T_{2n}$ defined in \eqref{set F} and \eqref{order leq}.
Hence, all $[A]\in\im(\tilde\eta_r\circ\rho^\imath_{r})$. Consequently, $\rho^\imath_{r}$ is surjective.
\end{proof}
We remark that the epimorphism was first obtained by Bao--Wang in \cite[Thm.~5.4]{BW}.

\section{A BLM type construction for $\bfUin$}
We are ready to give a new presentation for $\bfUin$ via a basis and multiplication formulas of basis elements by generators.

Consider the $\Q(\up)$-algebra of the direct product of $\SenrQ$:
$$\bsS^\imath(n):=\prod_{r\geq0}\SenrQ.$$
For all $A\in\Xienl$ and $\bfj\in\Z^{2n}$, write
$$A(\bfj)=(A(\bfj,r))_{r\in\mathbb N}.$$ 
The algebra homomorphisms in \eqref{rho3} and isomorphisms in  \eqref{teta} 
induce a homomorphism and isomorphism, respectively,
$$\aligned
\rho^\imath:=(\rho^\imath_{r})_{r\in\N}:&\bfUin\lra\prod_{r\geq0}\End_{\bsH(r)}(\Om_{2n}^{\otimes r})\\
\tilde\eta:=\Pi_{r\geq0}\tilde\eta_{r}:&\prod_{r\geq0}\End_{\bsH(r)}(\Om_{2n}^{\otimes r})\lra\bsSen.
\endaligned$$
Recall also the involution $\omega$ given in \eqref{omega}.
Corollary \ref{motivation ex} gives immediately the following result.

\begin{prop}\label{phii}
There is a $\mathbb{Q}(\up)$-algebra injective homomorphism
$$\phi^\imath:=\tilde\eta\circ\rho^\imath\circ\omega: \bfUin\lra \bsSen$$
such that $e_h\mapsto E^{\theta}_{h,h+1}({\bf 0})$, $f_h\mapsto E^{\theta}_{h+1,h}({\bf 0})$, $d_i^{\pm 1}\mapsto O(\pm \bfe_i)$, and $t\mapsto E^{\theta}_{n,n+1}({\bf 0})+O(-\bfe_n)$ for all $1\leq h\leq n-1$ and $1\leq i\leq n$.
 \end{prop}
 \begin{proof}Clearly, $\phi^\imath$ is an algebra homomorphism. Since $\rho^\imath$ is the restriction to $\iota(\bfUin)$ of the injective algebra homomorphism
$$\rho=(\rho_{r})_{r\in\N}:\bfU(\mathfrak{gl}_{2n})\lra\prod_{r\geq0}\End_{\bsH(\fS_r)}(\Om_{2n}^{\otimes r}),$$
where $\rho_{r}$ are given in \eqref{rho2}, it follows that $\phi^\imath$ is injective.
 \end{proof}
We now determine the image of $\phi^\imath$. Let $\Ain$ be the subspace of $\bsSen$ spanned by the linear independent set\footnote{See, e.g., \cite[Prop. 4.1]{DF} for one proof.}
$$\mathcal B=\{A(\bfj)\mid A\in\Xienl,\bfj\in\Z^{2n}\}.$$

\begin{theo}\label{epi} We have $\Ain=\im(\phi^\imath)$. Moreover, the algebra homomorphism $\phi^\imath$ induces an algebra isomorphism
$\phi^\imath:\bfUin\to\Ain$.
\end{theo}
\begin{proof} The proof for the first assertion is somewhat standard (see that of \cite[Thm. 5.2]{DW} or \cite[Lem.~5.5]{BLM}). We outline it as follows.

By Proposition \ref{phii}, $\im(\phi^\imath)$ is generated by 
$$E^{\theta}_{i,i+1}({\bf 0}), \; E^{\theta}_{i+1,i}({\bf 0}),\;  O(\pm \bfe_i)$$
 for all $1\leq i\leq n$ (Note that $E^{\theta}_{n,n+1}=E^{\theta}_{n+1,n}$). By Theorem \ref{longMF}, $\im(\phi^\imath)\subseteq\Ain$. We now prove the reverse inclusion.
 
 For $A\in\Xienl,(i,h,j)\in\mathscr T_{2n}$ as in \eqref{set F},  $(\phi^\imath)^{-1}\big((a_{i,j}E^{\theta}_{h+1,h})(\bfl)\big)=\begin{cases}f_h^{(a_{i,j})},&\text{if }h<n;\\
 e_{2n-h}^{(a_{i,j})},&\text{if }h>n,\end{cases}$ by Corollary \ref{DP}. Let
 $${}^{(a_{i,j})\!}g:=(\phi^\imath)^{-1}\big((a_{i,j}E^{\theta}_{n+1,n})(\bfl)\big).$$
 Note that ${}^{(a_{i,j})\!}g$ is not a divided power and $g:={}^{(1)\!}g=t-d_n^{-1}$ by Proposition \ref{phii}.
 
 For $A\in\Xienl$, define
\begin{equation}\label{MB}
\fkmAl:=\prod _{(i,h,j)\in(\mathscr T_{2n},\leq)}(\phi^\imath)^{-1}\big((a_{i,j}E^{\theta}_{h+1,h})(\bfl)\big)\in\bfUin,
\end{equation}
where the product is taken with respect to the order $\leq$ (cf. \cite[(5.0.3)]{DW}). We claim that
$$\scmAl:=\phi^\imath(\fkmAl)=\prod _{(i,h,j)\in(\mathscr T_{2n},\leq)}(a_{i,j}E^{\theta}_{h+1,h})(\bfl)=A(\bfl)+\text{a linear comb. of $B(\bfj)$ with }B\prec A.$$
By Theorem \ref{longMF} and Corollary \ref{DP}, $\scmAl\in\im(\phi^\imath)$ is a linear combination of $B(\bfj)$. We need to determine its leading term.

Let 
\begin{equation}\label{pi_r}
\pi_r:\boldsymbol\sS^\imath(n)\lra\SenrQ
\end{equation}
 be the canonical projection on $r$th component. Then
$$\pi_r(\scmAl)=\prod _{(i,h,j)\in(\mathscr T_{2n},\leq)}(a_{i,j}E^{\theta}_{h+1,h})(\bfl,r),$$
 Now a proof similar to that of \cite[Lemma 5.1]{DW} via Theorem \ref{TR} together Corollary \ref{DP} shows that
$$\prod _{(i,h,j)\in(\mathscr T_{2n},\leq)}(a_{i,j}E^{\theta}_{h+1,h})(\bfl,r)=A(\bfl,r)+\text{a linear combination of $C(\bfj,r)$ with }C\prec A,$$
proving the claim.

 Finally, by induction on $\|A\|:=\sum_{i=1}^{2n}{j-i\choose 2}(a_{i,j}+a_{j,i})$, a proof similar to that of \cite[Thm.~5.2]{DW} shows that every $A(\bfl)\in\im(\phi^\imath)$, and hence, all $A(\bfj)\in\im(\phi^\imath)$.
\end{proof}

By identifying $\bfUin$ with $\Ain$ via $\phi^\imath$, we now summarise our discovery so far in the following result. 

Let $\fkm^{A,\bfj}=O(\bfj)\scmAl=O(\bfe_1)^{j^*_1}\cdots O(\bfe_n)^{j^*_n}\scmAl$, where $(j^*_1,\ldots,j^*_n)=\bfj^*$ is defined in \eqref{j*}. 
Let $\Z^{n*}$ denote the subset of $\Z^{2n}$ consisting of $(\bfj,0^n)$ for all $\bfj\in\Z^n$.

\begin{theo}\label{summary} The $i$-quantum group $\bfUin$ has two bases
$$\aligned
\mathcal B&=\{A(\bfj)\mid A\in\Xienl,\bfj\in\Z^{2n}\}=\{A(\bfj)\mid A\in\Xienl,\bfj\in\Z^{n*}\} \text{ and }\\
\mathcal M&=\{\fkm^{A,\bfj}\mid A\in\Xienl,\bfj\in\Z^{2n}\}=\{\fkm^{A,\bfj}\mid A\in\Xienl,\bfj\in\Z^{n*}\}.
\endaligned$$
Furthermore, with the notation in \eqref{al-}, we may present $\bfUin$ by the basis $\mathcal B$ and the following multiplication formulas by generators $d_j,e_h,f_h,g=t-d_n^{-1}$:
\begin{itemize}
\item[] \vspace{-3ex}
$$(1)\;\;d_j\cdot A({\bf j})=\up^{\sum_{i=1}^{2n}a_{j,i}}A({\bf {j+\bfe_j}}),\quad A({\bf j})\cdot d_j=\up^{\sum_{i=1}^{2n}a_{i,j}}A({\bf {j+\bfe_j}});\qquad\qquad\quad\;\;\quad\;$$
\item[]\vspace{-3ex}
\begin{equation*}
\aligned
(2)\;\;e_h\cdot A({\bf j})&=\sum_{\substack{1\leq p<h\\ a_{h+1,p}\geq 1}}\up^{\b_p (A,h)}\overline{[\![a_{h,p}+1]\!]}
(A+E^{\theta}_{h,p}-E^{\theta}_{h+1,p})({\bf j}+\alpha_{h})\\
&+\varepsilon \frac{\up^{\b_h(A,h)-j_h-j_{2n+1-h}}}{\up-\up^{-1}}\Big((A-E^{\theta}_{h+1,h})({\bf j}+\al _h)-(A-E^{\theta}_{h+1,h})({\bf j}+\al^-_h)\Big)\\
&+ \up^{\b_{h+1}(A,h)+j_{h+1}+j_{2n-h}}\overline{[\![a_{h,h+1}+1]\!]}(A+E^{\theta}_{h,h+1})({\bf j})\\
&+\sum_{\substack{h+1<p\leq 2n\\ a_{h+1,p}\geq 1}}\up^{\b_p(A,h)}\overline{[\![a_{h,p}+1]\!]}
(A+E^{\theta}_{h,p}-E^{\theta}_{h+1,p})({\bf j})\\
\endaligned
\end{equation*}
where $\varepsilon=\delta^\leq_{1,a_{h+1,h}}$.
\item[]\vspace{-3ex}
\begin{equation*}
\aligned
(3)\; \;f_h\cdot A({\bf j})&=
\sum_{\substack{1\leq p<h\\ a_{h,p}\geq 1}}\up^{\b'_p(A,h)}\overline{[\![a_{h+1,p}+1]\!]}
(A-E^{\theta}_{h,p}+E^{\theta}_{h+1,p})({\bf j})\\
&+ \up^{\b'_h(A,h)+j_{h}+j_{2n+1-h}}\overline{[\![a_{h+1,h}+1]\!]}(A+E^{\theta}_{h+1,h})({\bf j})\\
&+\varepsilon'\frac{\up^{\b'_{h+1}(A,h)-j_{h+1}-j_{2n-h}}}{\up-\up^{-1}}\big((A-E^{\theta}_{h,h+1})({\bf j}-\al _h)-(A-E^{\theta}_{h,h+1})({\bf j}+\al^-_h)\big)\\
&+\sum_{\substack{h+1<p\leq 2n\\ a_{h,p}\geq 1}}\up^{\b'_p(A,h)}\overline{[\![a_{h+1,p}+1]\!]}
(A-E^{\theta}_{h,p}+E^{\theta}_{h+1,p})({\bf j}-\al_h),
\endaligned
\end{equation*}
where $\varepsilon'=\delta_{1,a_{h,h+1}}^\leq$ is given in \eqref{c_A}.
\item[]\vspace{-3ex}
 \begin{equation*}\aligned
(4)\;\;g\cdot& A({\bf j})=\sum_{i\neq n,n+1\atop a_{n,i}\geq1}\up^{\b'_i(A,n)-\delta^\leq_{n+1,i}}\overline{[\![a_{n+1,i}+1]\!]}(A-E^{\theta}_{n,i}+E^{\theta}_{n+1,i})({\bf j})\\
&\quad+\up^{\b'_n(A,n)+j_n+j_{n+1}}\overline{[\![a_{n+1,n}+1]\!]}(A+E^{\theta}_{n+1,n})({\bf j})\\
&\quad+\delta^\leq_{1,a_{n,n+1}}\frac{\up^{\b'_{n+1}(A,n)-j_{n+1}-j_n}}{\up-\up^{-1}}\Big((A-E^{\theta}_{n,n+1})({\bf j})-(A-E^{\theta}_{n,n+1})({\bf j}-\bfe^\theta_n)\Big)\\
&\quad+ c_AA({\bf j}-\bfe_n)\\
\endaligned
\end{equation*}
 where $c_A$ is defined in \eqref{c_A}.
\end{itemize}
\end{theo}
\begin{proof}We only make some comments on the first assertion. In the two descriptions for each basis,
the first ones have duplications since $A(\bfj)=A(\bfj^*)$ and $\fkm^{A,\bfj}=\fkm^{A,\bfj^*}$ where $\bfj^*$ is defined in \eqref{j*} and regarded as an element of $\sZ^{n*}$. The second ones are a more accurate description.\footnote{A similar description for \cite[Cor. 5.3]{DW} is also needed to avoid possible confusions.} The two descriptions giving the same basis set follows from the fact that the map $\Z^{2n}\to\Z^n, \bfj\mapsto\bfj^*$ is surjective.  One then uses a standard argument (see, e.g., the proof of \cite[Prop. 4.1]{DF}) to prove that the (second) set for $\mathcal B$
is linearly independent. Hence, the assertion for $\mathcal M$ follows,
\end{proof}
\begin{remark}
The presentation given in the theorem defines the regular representation of $\bfUin$. It would be interesting to know if the regular representation can be constructed directly via the $\up$-differential operator approach developed in \cite{DZ}.
\end{remark}

\section{Finite symplectic groups and quantum hyperalgebras of $\bfUin$}
We now bring finite symplectic groups into the game. As seen from \cite{BKLW}, the $q$-Schur algebra $\SenrZ$ or rather its $\sA$-form $\SenrA$, where $\sA=\Z[\bsq]$ ($\bsq=\up^2$), has a convolution algebra specialisation via a certain flag variety of type $C$. This specialisation links such algebras to representations  of finite symplectic groups in certain Harish-Chandra series at the cross-characteristic level. Now, the new construction of $\bfUin$ developed in \S7 allows us to extend further the link to representations of a certain hyperalgebra of $\bfUin$.


For any field $k$, let $\GL_{2n}(k)$ be the genernal linear group over $k$ and consider the group isomorphism
\begin{equation}\notag\label{vartheta}
\vartheta:\GL_{2n}(k)\longrightarrow \GL_{2n}(k),\;x\longmapsto J^{-1}(x^t)^{-1}J,
\end{equation}
where $J=\begin{pmatrix}0&I_n\\-I_n&0\end{pmatrix}$ with the $n\times n$ identity matrix $I_n$.
The symplectic group
$$\Sp_{2n}(k):=\{x\in\GL_{2n}(k)\mid J=x^tJx\}$$
is the fixed-point group of $\vartheta$.

Let  $G(q):=\Sp_{2r}(k)$ for $k=\mathbb F_q$, the finite field of $q$ elements. 

For $\la\in \Lambda(n,r)$, let $P_{\hla}(q)$ be the standard parabolic subalgebra of $\GL_{2r}(\mathbb F_q)$ associated with $\hla$, consisting of upper quasi-triangular matrices with blocks of sizes $\hla_i$ on the diagonal.  Let
$$P_\la(q)=P_{\hla}(q)\cap G(q).$$
Then $G(q)$ acts on the set $G(q)/P_\la(q)$ of left cosets $gP_\la(q)$ in $G(q)$. For any commutative ring $R$, this action induces a permutation representation over $R$ which is isomorphic to the induced representations $\text{Ind}_{P_{\la}(q)}^{G(q)}1_R$ of the trivial representation $1_R$ of $P_\la(q)$ to $G(q)$ 
and define
\begin{equation}\label{UPB}
\mathcal E_{q,R}(n,r)=\End_{RG(q)}\bigg(\bigoplus_{\la\in\Lambda(n,r)}\text{Ind}_{P_{\la}(q)}^{G(q)}1_R\bigg)^{\text{op}}.
\end{equation}
For any integral domain $R$ with $q\in R$, base change via the
specialisation $\sA\to R, \up^2\mapsto q$ induces an isomorphism
 $$\SenrR:=\SenrA\otimes_\sA R\cong \mathcal E_{q,R}(n,r).$$
See \cite[Prop. 6.6]{BKLW}\footnote{The algebra $\mathbf S^\imath$ in \cite{BKLW} is the centraliser subalgebra $\SnriZ$ given in \eqref{Sji}.} or \cite[Thm.~4.2]{LW}. 

\begin{remark} By Remark \ref{Sjnr}, we may also introduce the $q$-Schur algebra $\SnrA$ of type $B$ over $\mathcal A=\Z[\bsq]$ which has an interpretation similar to \eqref{UPB} via finite orthogonal groups $G(q)=\O_{2r+1}(q)$ with $P_\hla$ replaced by $P_\wla(q)$ and $\La(n,r)$ replaced by $\La(n+1,r)$; see \cite[(3.0.2)]{DW}. 
Here, for $\la=(\la_1,\ldots,\la_n,\la_{n+1})\in\La(n+1,r)$,
$\wla=(\la_1,\ldots,\la_n,2\la_{n+1}+1,\la_n,\ldots,\la_1).$
\end{remark}

As seen above, representations of $\SenrR$ is related to those of finite symplectic groups $G(q)$. If we can lift the epimorphism in Corollary \ref{rho4} to the integral level (i.e., a homomorphism from some quantum hyperalgebra $\UinR$ to $\SenrR$), then the representation category of $\SenrR$ is a full subcategory of that of $\UinR$. In this way, we establish a link between representations of $i$-quantum groups and finite symplectic groups in cross-characteristics.

To define $\UinR$, we need a candidate Lusztig type form $\UinZ$ of $\bfUin$. Traditionally,
$\UinZ$ is the $\ZZ$-subalgebra of $\bfUin$ generated by divided powers $e_{h}^{(m)}$, $f_{h}^{(m)}$, $t^{(m)}$ and $d_i$, $\begin{bmatrix}d_i;0\\s\end{bmatrix}$ for all $m,s\in\N$  , $1\leq h\leq n-1$ and $1\leq i\leq n$.  However, by Theorem \ref{summary}, $E^{\theta}_{n,n+1}(\bfl)O(-\bfe_n)=O(-\bfe_n)E^{\theta}_{n,n+1}(\bfl)$ and, by identifying $\bfUin$ as $\Ain$ under $\phi^\imath$, 
$$t^{(m)}=\frac1{[m]^!}(E^{\theta}_{n,n+1}(\bfl)+O(-\bfe_n))^m=\frac1{[m]^!}\sum_{j=0}^m{m\choose j}E^{\theta}_{n,n+1}(\bfl)^{m-j}O(-\bfe_n)^j.$$
Hence,
$$t^{(m)}=E^{\theta}_{n,n+1}(\bfl)^{(m)}+f_1E^{\theta}_{n,n+1}(\bfl)^{(m-1)}O(-\bfe_n)+\cdots+f_{m-1}E^{\theta}_{n,n+1}(\bfl)O(-\bfe_n)^{m-1}+f_mO(-\bfe_n)^m$$
for some $f_1,\ldots,f_m\in\Q(\up)$. These rational function coefficients in the display above and in Corollary \ref{DP} show that this form cannot be used as the image $\pi_{r}\circ\phi^\imath(\UinZ)$ cannot be inside $\SenrZ$. 

Motivated by the proof of Corollary \ref{motivation ex}, we make the following definition. 
\begin{defi} Let $\UinZ$ be the $\sZ$-subalgebra of $\bfUin$ generated by 
$$e_h^{(m)}, f_h^{(m)}, d_i, \begin{bmatrix}d_i;0\\s\end{bmatrix}, {}^{(m)\!}g,$$
for all $h,i\in[1,n]$ ($h\neq n$) and $m,s\in\N$.
\end{defi}

Note that, if  we identify $\bfUin$ with $\Ain$  as in Theorem \ref{summary}, then,
by Corollary \ref{DP}, the generators $e_h^{(m)}, f_h^{(m)},{}^{(m)\!}g$ can be unified as the generators
$(mE^{\theta}_{h,h+1})(\bfl), (mE^{\theta}_{h+1,h})(\bfl)$.

For $\la\in\N^n$ and $\tau\in\N_2^n$, where $\N_2=\{0,1\}$, let 
$$\bigg[{d\atop \la}\bigg]=\prod_{i=1}^n\begin{bmatrix}d_i;0\\\la_i\end{bmatrix},\quad
d^\tau=d_1^{\tau_1}d_2^{\tau_2}\cdots d_n^{\tau_n},$$
where $\big[{d_i;0\atop\la_i}\big]$ is defined in \eqref{Kbinom}.
Recall the elements $\fkm^{A,\bfl}$ defined in \eqref{MB} and the canonical projection $\pi_r$ in \eqref{pi_r}.
\begin{theo}The $\sZ$-algebra $\UinZ$ contains the basis
$$\mathcal M_\sZ=\bigg\{d^\tau\bigg[{d\atop \la}\bigg]\fkm^{A,\bf0}\mid A\in\Xienl,\tau\in\N_2^n,\la\in\N^n\bigg\}$$
for $\bfUin$. Hence, restricting the map $\phi^\imath$ in Proposition \ref{phii} to $\UinZ$
induces a surjective homomorphism
$$\phi_{r,\sZ}^\imath:=\pi_r\circ\phi^\imath:\UinZ\lra\SenrZ.$$
\end{theo}
\begin{proof}By definition, we have $\mathcal M_\sZ\subset\UinZ$. The basis claim follows from Theorem \ref{summary} and the fact that both $\{d^\bfj\mid\bfj\in\Z^n\}$ and $\{d^\tau\big[{d\atop \la}\big]\mid \tau\in\N_2^n,\la\in\N^n\}$ form bases for $\Q(\up)[d_1^{\pm1},\ldots,d_n^{\pm1}]$ (see, e.g., the proof for \cite[Thm.~14.20]{DDPW08}).
 
 The surjectivity assertion is seen from the proof of Corollary \ref{rho4} since, for any $A\in\Xi_{2n,2r}$ and $\la=\ro(A)$,
 we have, by \eqref{k over la}, $\phi^\imath_{r,\sZ}\big(\big[{d\atop\la}\big]\big)=\big[{k\atop\la}\big]=[\ro(A)]$ and, by Theorem \ref{TR},
 $$\aligned
 \phi_{r,\sZ}^\imath(\bigg[{d\atop \la}\bigg]\fkm^{A',\bfl})&=[\ro(A)]\cdot\prod _{(i,h,j)\in(\mathscr T_{2n},\leq)}(a_{i,j}E^{\theta}_{h+1,h})(\bfl,r)\\
 &=[A]+(\text{a $\sZ$-linear comb. of $[B]$ with }B\prec A),
 \endaligned$$
 where $A'$ is the matrix obtained from $A$ by replacing its diagonal with zeros.
Thus, the set $\{\phi_{r,\sZ}^\imath(\big[{d\atop \ro(A)}\big]\fkm^{A',\bfl})\mid A\in\Xi_{2n,2r}\}$ forms a basis for $\SenrZ$. Hence, the theorem is proved.
\end{proof}

Note that, if we identify $\SenrZ$ as a $\sZ$-form of $\End_{\bsH(r)}(\Omega_{2n}^{\otimes r})$, then $\phi_{r,\sZ}^\imath$ may be identified as the restriction of the map $\rho_r^\imath$ given in \eqref{rho3} to $\UinZ$.
\begin{coro}For any commutative ring $R$ which is a $\sZ$-algebra via $\up\mapsto \sqrt{q}\in R$, the $q$-Schur algebra $\SenrR$ is a homomorphic image
of $\UinR$. In particular, the category $\SenrR$-{\bf mod} of $\SenrR$-modules is a full subcategory of the category $\UinR$-{\bf mod} of $\UinR$-modules.
\end{coro}

\begin{remark}(1) If $R$ takes a member from a modular system $(\mathcal O,K,k)$, where $\mathcal O$ is a local DVR with fraction field $K$ and residue field $k$, and $q$ is a prime power and non-zero in $k$, representations of $\SenrR$ are closely related to the representations of finite symplectic group $G(q)$ in cross-characteristics, especially those in the unipotent principal series. For example, the decomposition matrix for $\SenrR$ is unitriangular and is part of the decomposition matrix $RG(q)$. See \cite[Chap. 5]{DPS} for more details. 

(2) The $\sZ$-subalgebra $\UinZ$ of $\bfUin$ contains the $\sZ$-subalgebra $U_\sZ(\mathfrak{gl}_n)$ having the same generators with all $(mE_{n,n+1}^\theta)(\bfl)$ removed. This subalgebra is the Lusztig form of $\bfU(\mathfrak{gl}_n)$ which is $\sZ$-free. It would be interesting to know if the set $\mathcal M_\sZ$ spans $\UinZ$ and hence, forms a basis for $\UinZ$. 
\end{remark}


\smallskip
{\bf Acknowledgement.} The authors would like to thank Weiqiang Wang and Yiqiang Li  for many helpful discussions during the writing of the paper.

 \end{document}